\documentclass[11pt,reqno]{amsart}
\usepackage{amsmath,amsxtra,latexsym,amsthm,amssymb,amscd,pb-diagram}
\usepackage{mathrsfs,mathabx,amsfonts}
\usepackage[colorlinks=true,linkcolor=red,citecolor=blue]{hyperref}
\usepackage[margin=1in]{geometry}

\usepackage[x11names]{xcolor}

\usepackage{bm}
\usepackage{color}
\usepackage{makecell,multirow,diagbox}
\usepackage{mathtools}
\usepackage[displaymath, mathlines]{lineno}

\usepackage[displaymath, mathlines]{lineno}
\usepackage{tikz}
\usepackage{graphicx}
\usepackage{epsfig}
\usepackage{array}
\usepackage{enumitem}
\usepackage{float}

\usepackage{graphicx}
\usepackage{epsfig}
\usepackage{array}
\usepackage{enumitem}

\newtheorem{definition}{Definition}[section]
\newtheorem{theorem}{Theorem}[section]
\newtheorem{proposition}{Proposition}[section]
\newtheorem{lemma}{Lemma}[section]

\newtheorem{remark}{Remark}[section]

\newcommand{\R}{\mathbb{R}}

\newcommand{\ity}{\infty}

\begin{document}
\title[Weakly coupled system of damped wave equations with mixed nonlinearities]{Critical curve for the weakly coupled system of damped wave equations with mixed nonlinearities}

\thanks{$^* $\textit{Corresponding author:} Dinh Van Duong (vanmath2002@gmail.com)}

\maketitle
\centerline{\scshape Dinh Van Duong$^{1,*}$, Tuan Anh Dao$^{1}$, Masahiro Ikeda$^{2, 3}$}
\medskip
{\footnotesize
	\centerline{$^1$ Faculty of Mathematics and Informatics, Hanoi University of Science and Technology}
	\centerline{No.1 Dai Co Viet road, Hanoi, Vietnam}
    \vspace{0.2cm}
    
   \centerline{$^2$ Center for Advanced Intelligence Project}
   \centerline{RIKEN 1-4-1, Nihonbashi,
Chuo-ku, Tokyo 103-0027, Japan}

\vspace{0.2cm}
\centerline{$^3$ Graduate School of Information Science and Technology, The University of Osaka}
\centerline{1-5 Yamadaoka, 565-0871, Suita, Osaka, Japan}
}
\medskip
    
\begin{abstract}
   In this paper, we would like to consider the Cauchy problem for a weakly coupled system of semi-linear damped wave equations with mixed nonlinear terms. Our main objective is to draw conclusions about the critical curve of this problem using tools from Harmonic Analysis. Precisely, we obtain a new critical curve $pq = 1+ 2/n$ for $n =1,2$ by proving global (in time) existence of small data Sobolev solutions when $pq > 1 +2/n$ and blow-up of weak solutions in finite time even for small data when $pq < 1+ 2/n$ for $n \geq 1$. From this, we infer the impact of the nonlinearities of time derivative-type on the critical curve associated with the system. 
\end{abstract}

\tableofcontents

\section{Introduction}
Let us consider the following Cauchy problem for weakly coupled system of semi-linear damped wave equations with mixed nonlinear terms:
\begin{equation} \label{Main.Eq.1}
\begin{cases}
u_{tt} -\Delta u + u_t= |v|^p, &\quad x\in \R^n,\, t > 0, \\
v_{tt} -\Delta v +v_t = |u_t|^q, &\quad x \in \mathbb{R}^n, \,t > 0,\\
(u, u_t, v, v_t)(0,x)= \varepsilon (u_0, u_1, v_0, v_1)(x), &\quad x \in \mathbb{R}^n, \\
\end{cases}
\end{equation}
where $p>1$ and $q>1$ stand for power exponents of nonlinearities and the positive constant $\varepsilon$ describes the size of initial data. The quantities $|v|^p$ and $|u_t|^q$ are respectively called power nonlinearity and nonlinearity of derivative-type. 

To get started, we now mention some historical views in terms of studying the following problem:
\begin{equation} \label{Main.Eq.2}
\begin{cases}
\phi_{tt}-\Delta \phi+ \phi_t= \mu |\phi|^p, &\quad x\in \R^n,\, t > 0, \\
(\phi, \phi_t)(0,x)= \varepsilon (\phi_0, \phi_1)(x), &\quad x\in \R^n,
\end{cases}
\end{equation}
where $p >1$ and $\mu \geq 0$. It is well known that the solutions to the damped wave equation behave like those of the heat equation as time goes to infinity. This is called \textit{diffusion phenomena} and investigated for a long time by many mathematicians (see \cite{Matsumura1976, Nishihara2003, MarcatiNishihara2003} and the references therein). Matsumura \cite{Matsumura1976} was the first to establish some basic decay estimates for the problem (\ref{Main.Eq.2}), also well-known as the linear damped wave equation for $\mu = 0$ and the semi-linear damped wave equation for $\mu = 1$. He concluded that the damped wave equation has a diffusive structure as $t \to \infty$. Since then, many papers have studied sharp $L^m-L^q$
  estimates (with $1 \leq m \leq q \leq \infty$) for the linear problem ($\mu = 0$), for example, \cite{Nishihara2003, DabbiccoEbert2017, Ikeda2019} and the references cited therein. For the semi-linear problem (\ref{Main.Eq.2}) with $\mu = 1$, the papers \cite{TodorovaYordanov2001, IkehataMiyaokaNakatake2004, IkehataTanizawa2005} are obtained a global (in time) weak solution to this problem with a power type of nonlinearity $|u|^p$ satisfying
  \begin{align}\label{local}
      \begin{cases}
      \vspace{0.3cm}
          1+ \displaystyle\frac{2}{n} < p < +\infty &\text{ if } n = 1,2,\\
          1 +\displaystyle\frac{2}{n} < p \leq \displaystyle\frac{n}{n-2} &\text{ if } n \geq 3,
      \end{cases}
  \end{align}
  by using weighted energy methods and assuming data in Sobolev spaces with additional regularity $L^1$. On the other hand, nonexistence of general global (in time) small data solutions is proved in \cite{TodorovaYordanov2001} for $1 < p < 1 +2/n$ and in \cite{Zhang2001} for $p = 1+2/n$. Therefore, the exponent $$p_F(n) := 1+\frac{2}{n} $$ 
 is referred to as the critical exponent of the semilinear problem (\ref{Main.Eq.2}) for $\mu = 1$, known as the Fujita exponent. It is also the critical exponent of the semi-linear heat equation $$\phi_t -\Delta \phi = |\phi|^p, \,\,\,\phi(0,x) = \phi_0(x), \quad t > 0,\,  x \in \mathbb{R}^n.$$
 Here, critical exponent is understood as the threshold
between global (in time) existence of small data weak solution and blow-up of solutions
even for small data. The diffusion phenomenon connecting the linear heat equation and the classical damped wave equation (see \cite{Narazaki2004, Nishihara2003}) sheds light on the parabolic nature of classical damped wave models with power-type nonlinearities, as seen through the decay estimates of their solutions. Additionally, to derive the critical regularity of nonlinear terms for the semilinear damped wave equation, the authors in \cite{Ebert2020} considered the equation of (\ref{Main.Eq.2}) with
nonlinearities $|u|^{1+\frac{2}{n}}\mu(|u|)$ on the right-hand side, where $\mu$ stands for a suitable
modulus of continuity.   

 Under additional regularity $L^m$ for the initial data, with $m \in [1,2]$, papers \cite{IkehataOhta2002, IkehataTanizawa2005} identified the critical exponent of problem (\ref{Main.Eq.2}) as $$p_c(m):= p_F\left(\frac{n}{m}\right) := 1 + \frac{2m}{n}.$$ However, the authors did not provide conclusions regarding the solution's properties when $p = p_c(m)$. Paper \cite{Ikeda2019} clarified this, showing that 
$m=1$ falls into the blow-up range, while $m \in (1, 2]$ belongs to the global existence range under certain specific conditions. Continuing the extension of the initial data belonging to Sobolev spaces of negative order $  \dot{H}^{-\gamma} \times \dot{H}^{-\gamma}, \gamma \geq 0,$ Chen-Reissig \cite{ChenReissig2023} determined that the critical exponent for problem (\ref{Main.Eq.1}) is 
        \begin{align*}
            p_c(2,\gamma) := p_{F}\left(\frac{n}{2}+\gamma\right) = 1+ \frac{4}{n+2\gamma}, 
        \end{align*}
        by proving the global existence with small data when $p > p_c(2,\gamma)$ and blow-up solution when $p < p_c(2,\gamma)$. However, the behavior of solutions to problem (\ref{Main.Eq.1}) with $p = p_c(2, \gamma)$
  remains an open problem. To clarify this, paper \cite{DuongDao2025} is to study problem (\ref{Main.Eq.2}) by imposing additional regularity $ \dot{H}^{-\gamma}_m $ on the initial data, with $m \in (1, 2], \gamma \geq 0 $ and the method used in this work is based on \cite{Ikeda2019}. Specifically, they determine the critical exponent for problem (\ref{Main.Eq.1}) to be
        \begin{align*}
            p_c(m, \gamma) := p_F\left(\frac{n}{m}+\gamma\right) = 1 + \frac{2m}{n+m\gamma}.
        \end{align*}
        Additionally, they show that the critical value $p = p_c(m, \gamma)$ belongs to the global existence range.
 
Based on the research results from the damped wave equation, we turn to the papers on  the weakly coupled system of damped wave equations
\begin{align}\label{Main.Eq.3}
    \begin{cases}
        u_{tt} -\Delta u + u_t = |v|^p, &\text{ } x \in \mathbb{R}^n,\, t > 0,\\
        v_{tt} - \Delta v + v_t = |u|^q, &\text{ } x \in \mathbb{R}^n,\, t > 0,\\
        (u,u_t,v,v_t)(0,x) = \varepsilon (u_0, u_1, v_0, v_1)(x), &\text{ } x \in \mathbb{R}^n.
    \end{cases}
\end{align}
It is claimed in \cite{SunWang2007, NishiharaWakasugi2014, NishiharaWakasugi2015, Takeda2009} that the \textit{critical curve} for system (\ref{Main.Eq.3}) is given by 
\begin{align*}
   \Gamma(p,q):= \frac{\max\{p, q\}+1}{pq-1} = \frac{n}{2}.
\end{align*}
Specifically, they proved that, if $\Gamma(p,q) < n/2$, then there exists a unique global (in time) small data Sobolev solution for all spatial dimensions. Sun-Wang \cite{SunWang2007} obtained the global existence result for (\ref{Main.Eq.3}) using the method due to Escobedo– Herrero and $L^p-L^q$ type estimates of the linear damped wave equation. Nishihara-Wakasugi \cite{NishiharaWakasugi2014} improved this result by proving a global existence result for all $n \geq 1$ by using the weighted energy method. This method enables them to treat the higher dimensional cases $n \geq 4$ with rapidly decaying (or compactly supported)
initial data. Furthermore, the weighted energy method is applicable to the case of damping with
time or space dependent coefficient. 
In the opposite case $\Gamma(p,q) \geq n/2$,  the aforementioned works have established that every non-trivial local (in time) weak solution, in generally
blows up in finite time by applying test function method, which has been developed by Mitidieri-Pohozaev \cite{MitidieriPohozaev2001, MitidieriPohozaev2009}. In particular, by constructing two test functions with suitable
scaling, we will derive a system of two nonlinear differential inequalities with their initial values. 

Problem (\ref{Main.Eq.1}) is constructed by modifying problem (\ref{Main.Eq.3}), replacing the nonlinear pair $(|v|^p, |u|^q)$ by $(|v|^p, |u_t|^q)$. To the best of the authors’ knowledge, there has been no previous work addressing the critical curve for this problem. For this reason, the main objective of this paper is to demonstrate that the curve
$$pq = 1 +\frac{2}{n}$$
serves as the critical curve for problem (\ref{Main.Eq.1}) in one and two spatial dimensions. Specifically, we prove that under the condition $pq > 1+2/n$ for $n=1,2$, along with certain additional assumptions, problem (\ref{Main.Eq.1}) admits a unique global Sobolev solution (in time) with small initial data. At this point, the crux of our approach is based on the technique of using
loss of decay associated with recently developed tools from Harmonic Analysis. The
advantage worthy of mentioning of allowing some loss of decay is to show how the
restrictions to the admissible exponents $p$ and $q$ could be relaxed. In addition, by employing the test function method, the authors also show that under the condition $pq < 1+2/n$, problem (\ref{Main.Eq.1}) exists no global weak solution, even for small initial data. Through this, we have shown that the nonlinear term $|u_t|^q$
  is indeed strong enough to shift the critical curve of system (\ref{Main.Eq.1}) closer to the point $(1,1)$; that is, the curve 
$\Gamma(p,q) = n/2$ actually lies above $pq = 1+2/n$. This represents a new finding in our study. For spatial dimensions $n \geq 3$, due to the limitations of the local existence condition (\ref{local}) and the strong regularity required of the initial data in the estimates for solutions to the classical linear damped wave equation, we are not yet able to characterize the sharp global existence region for this problem. However, we will try to investigate this in the near future.

\hspace{0.5cm}

\noindent \textbf{Notations} \medskip
\begin{itemize}[leftmargin=*]
\item We write $f\lesssim g$ when there exists a constant $C>0$ such that $f\le Cg$, and $f \sim g$ when $g\lesssim f\lesssim g$.

\item We denote $\widehat{w}(t,\xi):= \mathfrak{F}_{x\rightarrow \xi}\big(w(t,x)\big)$ as the Fourier transform with respect to the spatial variable of a function $w(t,x)$. Moreover, $\mathfrak{F}^{-1}$ represents the inverse Fourier transform.

\item As usual, $H_r^{a}$ and $\dot{H}_r^{a}$, with $r \in (1, \infty), a > 0$, denote Bessel and Riesz potential spaces based on $L^r$ spaces. Here $\big< \nabla\big>^{a}$ and $|\nabla|^{a}$ stand for the pseudo-differential operators with symbols $\big<\xi\big>^{a}$ and $|\xi|^{a}$, respectively.

\item For any $\gamma \in \R$, we denote by $[\gamma]^+:= \max\{\gamma,0\}$, its positive part. 
 
\item Finally, we introduce the following data spaces:
\begin{align*}
    \mathcal{D}_{q,p}^1(s) &:= \big(H^{s+1}\cap H_q^{|\frac{1}{q}-\frac{1}{2}|+1}\cap L^1\big) \times \left(H^s \cap L^q \cap L^1\right)\\
    &\qquad\qquad \times \left(H^s \cap L^p \cap L^1\right) \times \left(L^2 \cap L^p \cap L^1\right)
\end{align*}
and 
\begin{align*}
    \mathcal{D}_{\alpha, p}^2(s) &:= \big(H^{s+1} \cap H^{\frac{2}{\alpha}}_{\alpha} \cap L^1\big) \times \big(H^s \cap H^{\frac{1}{\alpha}-\frac{1}{2}}_{\alpha} \cap L^1\big)\\
    &\qquad \times \big(H^s \cap H^{|\frac{1}{p}-\frac{1}{2}|}_p \cap L^1\big) \times \left(H^{s-1} \cap L^p \cap L^1\right),
\end{align*}
with the assumptions of Theorems \ref{Theorem1} and \ref{Theorem3}, respectively, being satisfied.
\end{itemize}
\textbf{Main results.} Let us now state the main results which will be proved in this paper.
    \begin{theorem}[\textbf{Global existence in $1$D}]\label{Theorem1}
        Let $n=1$. Assume that the exponents $p, q$ satisfy $\min\{p,q\} > 1$ and the following condition:
        \begin{align}
            pq > 3 .\label{condition1.1}
        \end{align}
        Moreover, we also assume
        \begin{align}\label{condition1.1.2}
        \frac{1}{2} < s < \min\left\{1, \,\frac{3}{2}-\frac{1}{p},\, \frac{3}{2}-\frac{1}{q}\right\} 
        \end{align}
        and the initial data $(u_0, u_1, v_0, v_1) \in \mathcal{D}_{q,p}^1(s)$. Then, there exists a constant $\varepsilon_0 > 0$ such that for any $\varepsilon \in (0, \varepsilon_0]$, we have a unique global (in time) solution 
        \begin{align*}
            (u,v) \in \left(\mathcal{C}([0, \infty), L^2) \cap\mathcal{C}^1([0, \infty), H^s \cap L^q)\right) \times \mathcal{C}([0, \infty), H^s \cap L^p)
        \end{align*}
        to \eqref{Main.Eq.1} fulfilling the following estimates:
        \begin{align*}
            \|u_t(t,\cdot)\|_{L^2} &\lesssim \varepsilon (1+t)^{-\frac{5}{4}+[\gamma_1(p)]^+} \|(u_0, u_1, v_0, v_1)\|_{\mathcal{D}_{q,p}^1(s)},\\
            \|u_t(t,\cdot)\|_{L^q} &\lesssim \varepsilon (1+t)^{-\frac{1}{2}(1-\frac{1}{q})-1+ [\gamma_1(p)]^+} \|(u_0, u_1, v_0, v_1)\|_{\mathcal{D}_{q,p}^1(s)},\\
            \|u_t(t,\cdot)\|_{\dot{H}^s} &\lesssim \varepsilon (1+t)^{-\frac{5}{4}-\frac{s}{2} +[\gamma_1(p)]^+} \|(u_0, u_1, v_0, v_1)\|_{\mathcal{D}_{q,p}^1(s)},\\
            \|v(t,\cdot)\|_{L^p} &\lesssim \varepsilon (1+t)^{-\frac{1}{2}(1-\frac{1}{p})} \|(u_0, u_1, v_0, v_1)\|_{\mathcal{D}_{q,p}^1(s)},\\
            \|v(t,\cdot)\|_{L^2} &\lesssim \varepsilon (1+t)^{-\frac{1}{4}} \|(u_0, u_1, v_0, v_1)\|_{\mathcal{D}_{q,p}^1(s)},\\
            \| v(t,\cdot)\|_{\dot{H}^s} &\lesssim \varepsilon(1+t)^{-\frac{1}{4}-\frac{s}{2}}  \|(u_0, u_1, v_0, v_1)\|_{\mathcal{D}_{q,p}^1(s)},
        \end{align*}
        where $\gamma_1(p) := 1 -\frac{1}{2}(p-1) +\varepsilon_1$ with an arbitrarily small positive constant $\varepsilon_1$.
    \end{theorem}

     \begin{theorem}[\textbf{Global existence in $2$D}]\label{Theorem3}
        Let $n=2$. Assume that the exponents $p, q$ satisfy $\min\{p, q\} > 1$
        and the following condition:
        \begin{align}
            pq > 2 .\label{condition3.1}
        \end{align}
        Moreover, we also assume that
        \begin{align*}
            1 < s < \min\{2, p\}
        \end{align*}
        and the initial data $(u_0, u_1, v_0, v_1) \in \mathcal{D}_{\alpha, p}^2(s)$. Then, there exists a constant $\varepsilon_0 > 0$ such that for any $\varepsilon \in (0, \varepsilon_0]$,
        we have a unique global (in time) solution 
        \begin{align*}
            (u,v) \in \left(\mathcal{C}([0, \infty), L^2) \cap \mathcal{C}^1([0, \infty), H^{s} \cap L^\alpha) \right)\times \mathcal{C}([0, \infty), H^{s} \cap L^p)
        \end{align*}
        to \eqref{Main.Eq.1} fulfilling the following estimates:
        \begin{align*}
\|u_t(t,\cdot)\|_{L^\alpha} &\lesssim \varepsilon(1+t)^{-2+ \frac{1}{\alpha}+ [\gamma_2(p)]^+} \|(u_0, u_1, v_0, v_1)\|_{\mathcal{D}_{\alpha, p}^2(s)},\\
 \|u_t(t,\cdot)\|_{L^2} &\lesssim \varepsilon(1+t)^{-\frac{3}{2}+ [\gamma_2(p)]^+} \|(u_0, u_1, v_0, v_1)\|_{\mathcal{D}_{\alpha, p}^2(s)},\\
            \|u_t(t,\cdot)\|_{\dot{H}^{s}} &\lesssim \varepsilon(1+t)^{-\frac{3}{2}-\frac{s}{2} +[\gamma_2(p)]^+} \|(u_0, u_1, v_0, v_1)\|_{\mathcal{D}_{\alpha, p}^2(s)},\\
            \|v(t,\cdot)\|_{L^p} &\lesssim \varepsilon(1+t)^{-1+\frac{1}{p}} \|(u_0, u_1, v_0, v_1)\|_{\mathcal{D}_{\alpha, p}^2(s)},\\
            \|v(t,\cdot)\|_{L^2} &\lesssim \varepsilon(1+t)^{-\frac{1}{2}} \|(u_0, u_1, v_0, v_1)\|_{\mathcal{D}_{\alpha, p}^2(s)},\\
            \| v(t,\cdot)\|_{\dot{H}^{s}} &\lesssim \varepsilon(1+t)^{-\frac{1}{2}-\frac{s}{2}}  \|(u_0, u_1, v_0, v_1)\|_{\mathcal{D}_{\alpha, p}^2(s)},
        \end{align*}
        where we denote
        $$\alpha := \min \{2, q\} \quad \text{ and }\quad \gamma_2(p) := 2-p +\varepsilon_2$$
        with an arbitrarily small positive constant $\varepsilon_2$.
    \end{theorem}
    \begin{remark}
    \fontshape{n}
\selectfont
        It is obvious that two quantities $[\gamma_1(p)]^+$ and $[\gamma_2(p)]^+$ appearing in
Theorems \ref{Theorem1} and \ref{Theorem3} are nonnegative. Comparing the estimated solutions achieved for (\ref{Main.Eq.1}) with the corresponding ones for the linear equation (\ref{Main.Eq.2}) (see Lemmas \ref{LinearEstimates} and \ref{LinearEstimates_2D}),  we can understand that these quantities represent
some loss of decay for the norms of $u(t,\cdot)$ in a comparison with the corresponding equation. 
    \end{remark}

    \begin{theorem}[\textbf{Blow-up}]\label{Theorem2}
        Let $n \geq 1$ and $(u_0, u_1, v_0, v_1) \in (L^1)^4$ satisfying
        \begin{align}\label{condition2.1}
        \int_{\mathbb{R}^n} u_1(x) dx > 0 \text{ and } \int_{\mathbb{R}^n} \big(v_0(x) + v_1(x)\big) dx > 0.
        \end{align}
        We assume that the exponents $p,q$ satisfy $\min\{p, q\} > 1$ and the following condition:
        \begin{align}\label{condition2.2}
            pq < 1 + \frac{2}{n}.
        \end{align}
        Then, there is no global weak solution  $(u,v) \in \mathcal{C}^1([0, \infty), L^2) \times \mathcal{C}([0, \infty), L^2)$ to \eqref{Main.Eq.1}.
    \end{theorem}

\begin{remark}
\fontshape{n}
\selectfont
From the conditions (\ref{condition1.1}), (\ref{condition3.1}) and (\ref{condition2.2}) in Theorem \ref{Theorem1}, \ref{Theorem3} and \ref{Theorem2}, respectively, we claim that
the critical curve for (\ref{Main.Eq.1}) in the $p-q$ plane is precisely described by
\begin{align*}
    pq = 1 +\frac{2}{n}.
\end{align*}
For the purpose of observing more explicitly, let us illustrate some ranges
describing results for both global existence from Theorems \ref{Theorem1}-\ref{Theorem3} and blow-up from Theorem \ref{Theorem2} in the $p-q$ plane in the following figures:

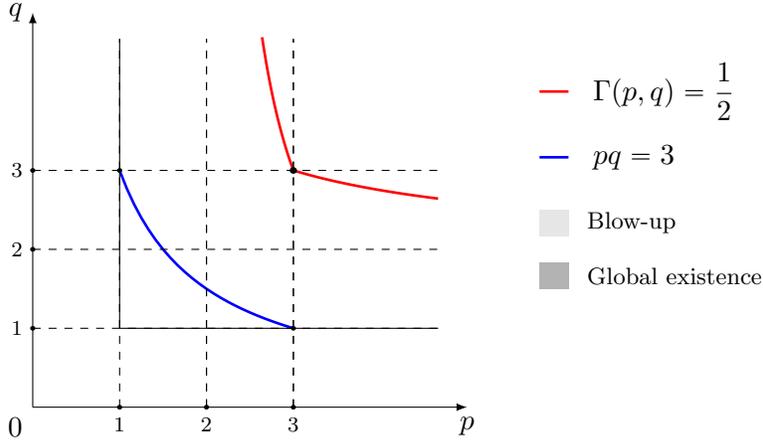
\begin{figure}[H]
\begin{center}
\begin{tikzpicture}[>=latex,xscale=0.9,scale=0.6]

\draw[->] (0,0) -- (7.5,0)node[below]{$p$};
\draw[->] (0,0) -- (0,7.5)node[left]{$q$};
\node[below left] at(0,0){$0$};
\node[below] at (2.5,1.0) {};


\draw[fill] (1.5,0) circle[radius=1pt];
\node[below] at (1.5,0){{\scriptsize $1$}};

\draw[fill] (3,0) circle[radius=1pt];
\node[below] at (3,0){{\scriptsize $2$}};
\draw[dashed] (3,0)--(3,7);

\draw[fill] (4.5,0) circle[radius=1pt];
\node[below] at (4.5,0)
{{\scriptsize $3$}};
\draw[dashed] (4.5,0)--(4.5,7);

\draw[fill] (0,1.5) circle[radius=1pt];
\node[left] at (0,1.5){{\scriptsize $1$}};
\draw[dashed] (0,1.5)--(7,1.5);

\draw[fill] (0,3) circle[radius=1pt];
\node[left] at (0,3){{\scriptsize $2$}};
\draw[fill] (0,4.5) circle[radius=1pt];
\node[left] at (0,4.5){{\scriptsize $3$}};

\fill[color=black!10!white]  (1.5,1.5)--(1.5,4.5)--(1.7, 3.970588)--(2, 3.375)--(2.3, 2.9347826)--(2.6, 2.5961538)--(3, 2.25)--(3.5, 1.928571)--(3.8, 1.7763157)--(4, 1.6875)--(4.5,1.5)--cycle;

\fill[color=black!30!white] (1.5,7)--(1.5,4.5)--(1.7, 3.970588)--(2, 3.375)--(2.3, 2.9347826)--(2.6, 2.5961538)--(3, 2.25)--(3.5, 1.928571)--(3.8, 1.7763157)--(4, 1.6875)--(4.5,1.5)--(7,1.5)--(7,7)--cycle;

\draw[domain = 1.5:4.5, blue,line width=1.0pt]plot(\x,{6.75/\x});
\draw[domain = 3.96:4.5, red,line width=1.0pt]plot(\x,{6.75/(\x-3)});

\draw[domain = 4.5:7, red,line width=1.0pt]plot(\x,{3+6.75/\x});

\draw[fill] (4.5,4.5) circle[radius=1.5pt];

\draw[thin] (1.5,1.5)--(1.5,7);
\draw[thin] (1.5,1.5)--(7,1.5);
\draw[dashed] (1.5,0)--(1.5,7);
\draw[dashed] (3,0)--(3,7);
\draw[dashed] (4.5,0)--(4.5,7);
\draw[fill] (4.5,1.5) circle[radius=1pt];

\draw[dashed] (0,4.5)--(7,4.5);

\draw[dashed] (0,3)--(7,3);
\draw[fill] (1.5,4.5) circle[radius=1pt];

\fill[color=black!10!white] (9.25,3.75)--(8.75,3.75)--(8.75,3.25)--(9.25,3.25)--cycle;
\node[right] at (9.4,3.5) {{\footnotesize \text{Blow-up}}};

\fill[color=black!30!white] (9.25,2.75)--(8.75,2.75)--(8.75,2.25)--(9.25,2.25)--cycle;
\node[right] at (9.4,2.5) {{\footnotesize \text{Global existence}}};

\draw[thin, color=blue,line width=1.0pt] (9.25,4.75)--(8.75,4.75);
\node[right] at (9.5,4.75) {$pq=3$};

\draw[thin, color=red,line width=1.0pt] (9.25,6)--(8.75,6);
\node[right] at (9.5,6) {$\Gamma(p,q)=\displaystyle\frac{1}{2}$};

\end{tikzpicture}
\caption{Global existence and blow-up results in the $p-q$ plane when $n=1$.}
\label{fig.zone1}
\end{center}
\end{figure}

\begin{figure}[H]
\begin{center}
\begin{tikzpicture}[>=latex,xscale=0.9,scale=0.6]

\draw[->] (0,0) -- (7.5,0)node[below]{$p$};
\draw[->] (0,0) -- (0,7.5)node[left]{$q$};
\node[below left] at(0,0){$0$};
\node[below] at (2.5,1.0) {};


\draw[fill] (1.5,0) circle[radius=1pt];
\node[below] at (1.5,0){{\scriptsize $1$}};

\draw[fill] (3,0) circle[radius=1pt];
\node[below] at (3,0){{\scriptsize $2$}};
\draw[dashed] (3,0)--(3,7);

\draw[fill] (4.5,0) circle[radius=1pt];
\node[below] at (4.5,0)
{{\scriptsize $3$}};
\draw[dashed] (4.5,0)--(4.5,7);

\draw[fill] (0,1.5) circle[radius=1pt];
\node[left] at (0,1.5){{\scriptsize $1$}};
\draw[dashed] (0,1.5)--(7,1.5);

\draw[fill] (0,3) circle[radius=1pt];
\node[left] at (0,3){{\scriptsize $2$}};
\draw[fill] (0,4.5) circle[radius=1pt];
\node[left] at (0,4.5){{\scriptsize $3$}};

\fill[color=black!10!white]  (1.5,3)--(1.7, 2.647058)--(2, 2.25)--(2.3, 1.9565217)--(2.6, 1.730769)--(3,1.5)--(1.5,1.5)--cycle;

\fill[color=black!30!white] (1.5,3)--(1.7, 2.647058)--(2, 2.25)--(2.3, 1.9565217)--(2.6, 1.730769)--(3, 1.5)--(7,1.5)--(7,3)--cycle;

\fill[color=black!30!white](1.5,3)--(7,3)--(7,7)--(1.5,7)--cycle;

\draw[domain = 1.5:3, blue,line width=1.0pt]plot(\x,{4.5/\x});

\draw[domain = 2.14:3, red,line width=1.0pt]plot(\x,{4.5/(\x-1.5)});

\draw[domain = 3:7, red,line width=1.0pt]plot(\x,{1.5 +4.5/\x});

\draw[fill] (3,3) circle[radius=1.5pt];

\draw[thin] (1.5,1.5)--(1.5,7);
\draw[thin] (1.5,1.5)--(7,1.5);
\draw[dashed] (1.5,0)--(1.5,7);
\draw[dashed] (3,0)--(3,7);
\draw[dashed] (4.5,0)--(4.5,7);
\draw[fill] (4.5,1.5) circle[radius=1pt];

\draw[dashed] (0,4.5)--(7,4.5);

\draw[dashed] (0,3)--(7,3);
\draw[fill] (1.5,4.5) circle[radius=1pt];

\fill[color=black!10!white] (9.25,3.75)--(8.75,3.75)--(8.75,3.25)--(9.25,3.25)--cycle;
\node[right] at (9.4,3.5) {{\footnotesize \text{Blow-up}}};

\fill[color=black!30!white] (9.25,2.75)--(8.75,2.75)--(8.75,2.25)--(9.25,2.25)--cycle;
\node[right] at (9.4,2.5) {{\footnotesize \text{Global existence}}};

\draw[thin, color=blue,line width=1.0pt] (9.25,4.75)--(8.75,4.75);
\node[right] at (9.5,4.75) {$pq=2$};

\draw[thin, color=red,line width=1.0pt] (9.25,6)--(8.75,6);
\node[right] at (9.5,6) {$\Gamma(p,q)=1$};

\end{tikzpicture}
\caption{Global existence and blow-up results in the $p-q$ plane when $n=2$.}
\label{fig.zone2}
\end{center}
\end{figure}
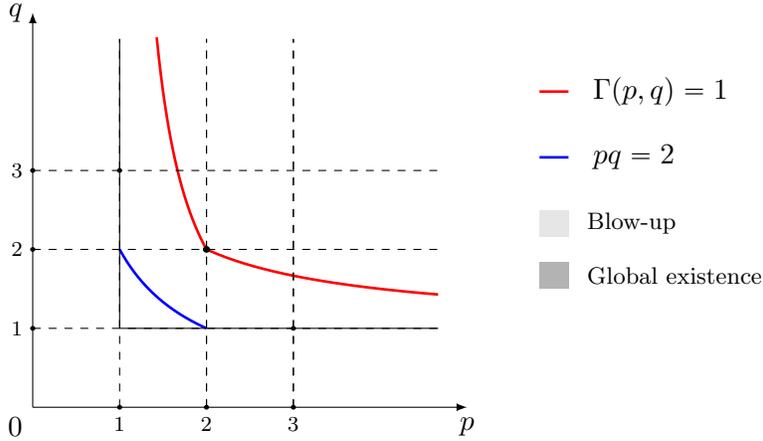
The above figures clearly indicate that substituting the nonlinear term $|u|^q$
  with $|u_t|^q$
  in system (\ref{Main.Eq.3}) results in the critical curve moving closer to the origin $(1,1)$.
\end{remark}

\textbf{This paper is organized as follows:} In Section \ref{section3}, we provide the proof of global (in time) existence results for solutions to the problem (\ref{Main.Eq.1}). Next, in Section \ref{Proof of blow-up results}, we establish the blow-up result in the subcritical case as well. Finally, we will present some open problems in Section \ref{Final Remarks}.
    
\section{ Global existence of small data solution}\label{section3}

\subsection{Philosophy of our approach}\label{Philosophy}
First, we can write the solution to linear problemm $(\ref{Main.Eq.2})$ with $\mu=0$ by the formula
\begin{equation*}
\phi^{\rm lin} :=\phi^{\rm lin}(t,x) = \varepsilon (\mathcal{K}(t,x) +\partial_t \mathcal{K}(t,x))\ast_x \phi_0(x) + \varepsilon \mathcal{K}(t,x) \ast_x \phi_1(x),
\end{equation*}
so that the solution to $(\ref{Main.Eq.1})$ becomes

\begin{align*}
\begin{cases}
\vspace{0.4cm}
    u(t,x) = u^{\rm lin}(t,x)+ u^{\rm non}(t,x) := u^{\rm lin}(t,x) +  \displaystyle\int_0^t \mathcal{K}(t-\tau,x) \ast_x |v(\tau,x)|^p d\tau,\\
    v(t,x) = v^{\rm lin}(t,x) + v^{\rm non}(t,x) :=   v^{\rm lin}(t,x) +\displaystyle\int_0^t \mathcal{K}(t-\tau,x) \ast_x |u_t(\tau,x)|^q d\tau,
    \end{cases}
\end{align*}
thanks to Duhamel's principle. Here
\begin{align*}
    \mathcal{K}(t,x) :=  \begin{cases}
        \vspace{0.3cm}\frak{F}^{-1}\left(\displaystyle\frac{e^{-\frac{t}{2}}\sinh{\left(t \sqrt{\frac{1}{4} -|\xi|^2}\right)}}{\sqrt{\frac{1}{4}- |\xi|^2}} \right)(t,x) &\text{ if } |\xi| \leq \displaystyle\frac{1}{2},\\
        \displaystyle\frak{F}^{-1}\left(\frac{e^{-\frac{t}{2}}\sin{\left(t \sqrt{|\xi|^2-\frac{1}{4}}\right)}}{\sqrt{|\xi|^2-\frac{1}{4}}}\right)(t,x) &\text{ if } |\xi| > \displaystyle\frac{1}{2}.
    \end{cases}
\end{align*}
Let $\chi_k= \chi_k(r)  \in \mathcal{C}^{\infty}([0, \infty))$ with $k\in\{\rm L,H\}$ be smooth cut-off functions having the following properties:
\begin{align*}
&\chi_{\rm L}(r)=
\begin{cases}
1 &\quad \text{ if }r\le \varepsilon^*/2, \\
0 &\quad \text{ if }r\ge \varepsilon^*,
\end{cases}
\text{ and } \qquad
\chi_{\rm H}(r)= 1 -\chi_{\rm L}(r).
\end{align*}
It is obvious to see that $\chi_{\rm H}(r)= 1$ if $r \geq \varepsilon^*$ and $\chi_{\rm H}(r)= 0$ if $r \le \varepsilon^*/2$.\\

\subsection{Proof of Theorem \ref{Theorem1}} 

To begin, we restate the important result as follows, the result related to the linear estimates in the one-dimensional case is restated in the following lemma.

\begin{lemma}[Linear Estimates in 1D]\label{LinearEstimates}
    Let $n=1$, $1 \leq \rho_2 \leq \rho_1 < \infty$ with $\rho_1 \ne 1$ and $0 \leq d_2 \leq d_1$. Then, the following estimates hold for all $t > 0$ (see Theorem 1.1 in \cite{Ikeda2019}):
    \begin{align*}
        &\| |\nabla|^{d_1} \mathcal{K}(t,x)\ast_x \varphi(x)\|_{L^{\rho_1}}\\
        &\qquad\lesssim (1+t)^{-\frac{1}{2}(\frac{1}{\rho_2}-\frac{1}{\rho_1})-\frac{d_1-d_2}{2}} \| |\nabla|^{d_2} \chi_{\rm L}(|\nabla|) \varphi\|_{L^{\rho_2}} + e^{-ct} \| |\nabla|^{d_1} \chi_{\rm H}(|\nabla|) \varphi\|_{H_{\rho_1}^{-1}},\\ 
        &\| |\nabla|^{d_1} \partial_t \mathcal{K}(t,x)\ast_x \varphi(x)\|_{L^{\rho_1}} \\
        &\qquad\lesssim (1+t)^{-\frac{1}{2}(\frac{1}{\rho_2}-\frac{1}{\rho_1})-\frac{d_1-d_2}{2}-1} \| |\nabla|^{d_2} \chi_{\rm L}(|\nabla|) \varphi\|_{L^{\rho_2}} + e^{-ct} \| |\nabla|^{d_1} \chi_{\rm H}(|\nabla|) \varphi\|_{L^{\rho_1}},
    \end{align*}
    where $c$ is a suitable positive constant. Moreover, we also obtain the following estimates (see Propositions 4.1 and 4.2 in \cite{DabbiccoEbert2017}):
    \begin{align*}
        &\||\nabla|^{d_1} \partial_t^2 \mathcal{K}(t,x) \ast_{x} \varphi(x) \|_{L^{\rho_1}} \\
        &\qquad\lesssim (1+t)^{-\frac{1}{2}(\frac{1}{\rho_2}-\frac{1}{\rho_1})-\frac{d_1}{2}-2} \|\chi_{\rm L}(|\nabla|)\varphi\|_{L^{\rho_2}} + e^{-ct} \|\chi_{\rm H}(|\nabla|)\varphi\|_{H_{\rho_1}^{\delta(\rho_1)+d_1+1}},  
    \end{align*}
    where
    \begin{align*}
    \delta(\rho_1) := \left|\frac{1}{\rho_1}-\frac{1}{2}\right|.
\end{align*}
\end{lemma}
\begin{remark}
\fontshape{n}
\selectfont
    In this paper, the third estimate in Lemma \ref{LinearEstimates} is used exclusively to identify the appropriate data space for establishing Theorem \ref{Theorem1}. The first and second estimates, on the other hand, will be employed frequently in our calculations.
\end{remark}

 Under the assumptions of Theorem \ref{Theorem1},  we define the following function spaces for all $T > 0$:
\begin{align*}
    X_{11}(T) &:=   L^\infty([0, T], H^s \cap L^q), \quad
    X_{12}(T) := L^\infty([0, T], H^s \cap L^p)
\end{align*}
and
\begin{align*}
    X_1(T) := X_{11}(T) \times X_{12}(T),
\end{align*}
with the norms
\begin{align*}
\|\varphi_{1}\|_{X_{11}(T)} &:= \sup_{t \in [0, T]}\bigg\{(1+t)^{\frac{1}{2}(1-\frac{1}{q})+1-[\gamma_1(p)]^+}\|\varphi_1(t,\cdot)\|_{L^q} \\
&\hspace{2cm}+ (1+t)^{\frac{5}{4}-[\gamma_1(p)]^+}\|\varphi_1(t,\cdot)\|_{L^2} + (1+t)^{\frac{5}{4}+\frac{s}{2}-[\gamma_1(p)]^+} \| \varphi_1(t,\cdot)\|_{\dot{H}^s}\bigg\},\\
\|\varphi_2\|_{X_{12}(T)} &:= \sup_{t \in [0, T]}\left\{(1+t)^{\frac{1}{2}(1-\frac{1}{p})}\|\varphi_2(t, \cdot)\|_{L^p} + (1+t)^{\frac{1}{4}} \|\varphi_2(t,\cdot)\|_{L^2}+ (1+t)^{\frac{1}{4}+\frac{s}{2}}  \|\varphi_2(t,\cdot)\|_{\dot{H}^s}\right\}
\end{align*}
and 
\begin{align*}
    \|(\varphi_1, \varphi_2)\|_{X_1(T)} := \|\varphi_1\|_{X_{11}(T)} + \|\varphi_2\|_{X_{12}(T)},
\end{align*}
respectively.
In addition, we also introduce the following function spaces for all $T > 0$:
\begin{align*}
    Y_{11}(T) := L^\infty([0, T], L^q), \quad Y_{12}(T) :=  L^\infty([0, T], H^s \cap L^p),
\end{align*}
and
\begin{align*}
    Y_1(T) := Y_{11}(T) \times Y_{12}(T),
\end{align*}
with the norms
\begin{align*}
    \|\varphi_1\|_{Y_{11}(T)} := &\sup _{t \in [0,T]} \bigg\{(1+t)^{1-\frac{1}{2q}+\frac{1}{2p}-[\gamma_1(p)]^+}\|\varphi_1(t,\cdot)\|_{L^q} \bigg\},\\
    \|\varphi_2\|_{Y_{12}(T)} := &\sup_{t \in [0, T]}\bigg\{ \|\varphi_2(t, \cdot)\|_{L^p} + (1+t)^{\frac{1}{2}(\frac{1}{p}-\frac{1}{2})}\|\varphi_2(t,\cdot)\|_{L^2}+ (1+t)^{\frac{1}{2}(\frac{1}{p}-\frac{1}{2}+s)}\|\varphi_2(t,\cdot)\|_{\dot{H}^s}\bigg\}
\end{align*}
and
\begin{align*}
    \|(\varphi_1, \varphi_2)\|_{Y_1(T)} := \|\varphi_1\|_{Y_{11}(T)} + \|\varphi_2\|_{Y_{12}(T)},
\end{align*}
respectively.
Next, we note that
\begin{align*}
    Z_1(T) := \{\varphi \text{ is a first-order differentiable  (in time) function on $[0,T]$ and } \varphi_t \in X_{11}(T)\}
\end{align*}
and
\begin{align*}
X_{11}(T, M) &:= \big\{\varphi_1 \in X_{11}(T), \|\varphi_1\|_{X_{11}(T)} \leq M\big\},\\
    X_{12}(T, M) &:= \big\{\varphi_2 \in X_{12}(T), \|\varphi_2\|_{X_{12}(T)} \leq M\big\},\\
    X_1(T,M) &:= \big\{(\varphi_1, \varphi_2) \in X_1(T), \|(\varphi_1, \varphi_2)\|_{X_1(T)} \leq M \},
\end{align*}
for all $T, M > 0$. We now proceed to the following important property.
\begin{lemma}\label{lemma1.2}
According to the assumptions of Theorem \ref{Theorem1},  $X_{11}(T, M)$  is a closed subset of $Y_{11}(T)$ with respect to the metric $Y_{11}(T)$. Moreover,  $X_{12}(T, M)$  is also a closed subset of $Y_{12}(T)$ with respect to the metric $Y_{12}(T)$
\end{lemma}
\begin{proof}
Firstly, we can easily see that $X_{11}(T) \subset Y_{11}(T)$ and $X_{12}(T) \subset Y_{12}(T)$ for all $T > 0$. As the proofs of each statement are entirely similar, we restrict ourselves to presenting the proof for
$X_{11}(T,M)$. Therefore, it suffices to show that
if a sequence in $X_{11}(T, M)$ converging in $Y_{11}(T)$, its limit will belong to $X_{11}(T, M)$. More specifically, we assume that $\{\varphi_j\}_{j=1}^{\infty} \subset X_{11}(T,M)$ and $\varphi_j \to \varphi \in Y_{11}(T)$ as $j \to \infty$. It is a fact that
    \begin{align*}
        L^\infty\left([0,T], H^s \cap L^q \right) = \left(L^1([0,T], H^{-s}  + L^{q'})\right)^*,
    \end{align*}
    where $q' = q/(q-1)$. Due to the separability of $L^1([0,T], H^{-s} + L^{q'})$, we apply Banach-Alaoglu theorem (see Theorem 3.16 or Corollary 3.30 in \cite{Brezis2011}) to take a subsequence $\{\varphi_{j(m)}\}_{m=1}^\infty$ and $\psi \in L^\infty\left([0,T], H^s  \cap L^q \right) $ such that
    \begin{align*}
        \varphi_{j(m)} \overset{*}{\rightharpoonup} \psi \quad\text{ as } \quad m \to \infty .
    \end{align*}
    In addition, it is a fact that
    \begin{align*}
        \|\psi\|_{X_{11}(T)} \leq \liminf_{m\to \infty} \|\varphi_{j(m)}\|_{X_{11}(T)} \leq M, 
    \end{align*}
    that is, $\psi \in X_{11}(T, M)$.
    On the other hand, both $\{\varphi_j\}_{j=1}^\infty$
and $\{\varphi_{j(m)}\}_{m=1}^{\infty}$
converge in the space of the distribution $
\mathcal{D}'([0, T] \times \mathbb{R}^n)$, that is,
\begin{align*}
    \varphi_{j(k)} &\to \psi \in \mathcal{D}'([0, T] \times \mathbb{R}^n) \quad\text{ as }\quad k \to \infty,\\
    \varphi_{j} &\to \varphi \in \mathcal{D}'([0, T] \times \mathbb{R}^n) \quad\text{ as } \quad m \to \infty.
\end{align*}
As a result, the uniqueness of the limit of distribution implies $\psi \equiv \varphi, $ which shows $\varphi \in X_{11}(T,M)$. 
\end{proof}
Next, we consider the operators $\mathcal{N}_{11}$ and $\mathcal{N}_{12}$ on the spaces $Z_1(T)$ and $X_{12}(T)$, respectively, as follows:
\begin{align*}
    \mathcal{N}_{11}: u \in Z_1(T) \to \mathcal{N}_{11}[u] := u^{\rm lin} + u^{\rm non}, \quad  \mathcal{N}_{12}: v\in X_{12}(T) \to \mathcal{N}_{12}[v] := v^{\rm lin} + v^{\rm non}.
\end{align*}
In order to prove Theorem \ref{Theorem1}, we will show a pair $(u^*, v^*)$ that satisfies the following integral equation
\begin{align*}
    (u^*, v^*) = (\mathcal{N}_{11}[u^*], \mathcal{N}_{12}[v^*]),
\end{align*}
for all $T > 0$.
Now, we are proving the following auxiliary results.
\begin{lemma}\label{Lemma2.3}
    Under the assumptions of Theorem \ref{Theorem1}, the following estimates hold for all $(u, v) \in Z_1(T) \times X_{12}(T)$, $\nu \geq 1$ and $T > 0$: 
    \begin{align}
        \| |u_t(t,\cdot)|^q  \|_{L^{\nu}} &\lesssim (1+t)^{-q-\frac{1}{2}(q-\frac{1}{\nu})+q[\gamma_1(p)]^+} \|u_t\|_{X_{11}(T)}^q\label{es3.1}
        \end{align}
        and
        \begin{align}
        \| |v(t,\cdot)|^p \|_{L^{\nu}} &\lesssim (1+t)^{-\frac{1}{2}(p-\frac{1}{\nu})} \|v\|_{X_{12}(T)}^{p},\label{es3.2}\\
        \||v(t,\cdot)|^p \|_{\dot{H}^s} &\lesssim (1+t)^{-\frac{1}{2}(p-\frac{1}{2})-\frac{s}{2}+\frac{\varepsilon_1}{2}} \|v\|_{X_{12}(T)}^{p} \label{es3.3},
    \end{align}
    where $\varepsilon_1$ is an arbitrarily small positive constant.
\end{lemma}
\begin{proof}

    Applying Proposition \ref{fractionalGagliardoNirenberg} we obtain 
    \begin{align*}
        \||u_t(t,\cdot)|^q\|_{L^{\nu}} = \|u_t(t,\cdot)\|_{L^{\nu q}}^q &\lesssim \|u_t(t,\cdot)\|_{L^q}^{q(1-\theta_{1})} \|u_t(t,\cdot)\|_{\dot{H}^s}^{q\theta_{1}}\\
        &\lesssim (1+t)^{-q-\frac{1}{2}(q-\frac{1}{\nu})+q[\gamma_1(p)]^+} \|u_t\|_{X_{11}(T)}^q,\\
        \||v(t,\cdot)|^p\|_{L^\nu} = \|v(t,\cdot)\|_{L^{\nu p}}^p &\lesssim \|v(t,\cdot)\|_{L^p}^{p(1-\theta_{2})} \| v(t,\cdot)\|_{\dot{H}^s}^{p\theta_{2}}\\
        &\lesssim (1+t)^{-\frac{1}{2}(p-\frac{1}{\nu})} \|v\|_{X_{12}(T)}^p,
    \end{align*}
     where
     $$\theta_{1} = \displaystyle\frac{\frac{1}{q}-\frac{1}{\nu q}}{\frac{1}{q}-\frac{1}{2}+s} \in [0,1]\quad \text{ and }\quad \theta_{2} = \displaystyle\frac{\frac{1}{p}-\frac{1}{\nu p}}{\frac{1}{p}-\frac{1}{2}+s} \in [0,1]. $$
     These conditions are satisfied due to $s\in (1/2,1)$. We immediately obtain the estimates (\ref{es3.1}) and (\ref{es3.2}). Next, we prove the estimate (\ref{es3.3}). Using Propositions \ref{FractionalPowers} and \ref{Embedding} with $s' < 1/2 < s$, it is a fact that
     \begin{align}
         \| |v(t,\cdot)|^p\|_{\dot{H}^s} &\lesssim \|v(t,\cdot)\|_{\dot{H}^s}\left(\|v(t,\cdot)\|_{\dot{H}^{s'}} + \|v(t,\cdot)\|_{\dot{H}^s}\right)^{p-1}. \label{es3.4}
     \end{align}
     From the norm definition of the space $X_{12}(T)$, we immediately obtain
     \begin{align}
         \|v(t,\cdot)\|_{\dot{H}^s} \lesssim (1+t)^{-\frac{1}{4}-\frac{s}{2}} \|v\|_{X_{12}(T)}^p. \label{es3.5}
     \end{align}
     Moreover, using again Proposition \ref{fractionalGagliardoNirenberg} we arrive at
     \begin{align}
         \|v(t,\cdot)\|_{\dot{H}^{s'}} &\lesssim \|v(t,\cdot)\|_{L^2}^{1-\frac{s'}{s}} \|v(t,\cdot)\|_{\dot{H}^s}^{\frac{s'}{s}} \lesssim (1+t)^{-\frac{1}{2}+\frac{\varepsilon_1}{2(p-1)}}\|v\|_{X_{12}(T)},\label{es3.6}
     \end{align}
     where we choose $s' = 1/2 -\varepsilon_1/(p-1)$ with $\varepsilon_1$ is an arbitrarily small positive constant. From (\ref{es3.4})-(\ref{es3.6}), we can conclude the estimate (\ref{es3.3}).
\end{proof}

\begin{proposition}\label{proposition3.1}
    Under the assumptions of Theorem \ref{Theorem1}, the following estimates hold for all $(u, v) \in Z_1(T) \times X_{12}(T)$:
    \begin{align}
    \|\partial_t u^{\rm non}\|_{X_{11}(T)} &\lesssim \|v\|_{X_{12}(T)}^{p} \label{Esti.Pro3.1.1}\\
    \|v^{\rm non}\|_{X_{12}(T)} &\lesssim \|u_t\|_{X_{11}(T)}^q. \label{Esti.Pro3.1.2}
    \end{align}
\end{proposition}
\begin{proof}
To begin with, we recall the definition of $u^{\rm non}$ and $v^{\rm non}$ as follows:
\begin{align*}
\begin{cases}
\vspace{0.2cm}
    u^{\rm non}(t,x) &= \displaystyle\int_0^t \mathcal{K}(t-\tau, x) \ast_x |v(\tau,x)|^p d\tau,\\
    v^{\rm non}(t,x) &= \displaystyle\int_0^t \mathcal{K}(t-\tau, x) \ast_x |u_t(\tau,x)|^p d\tau.
    \end{cases}
\end{align*}
   \textbf{First, we prove the estimate (\ref{Esti.Pro3.1.1})}. Fixing $\mu_1 \in \{2,q\}$ and applying Lemma \ref{LinearEstimates} with $\rho_2 = 1, \rho_1 = \mu_1$, $(d_1 , d_2) = (0,0)$ for $\tau \in [0, t/2)$, and $\rho_2=\rho_1 =\mu_1$, $(d_1 , d_2)=(0, 0) $ for $\tau \in [t/2, t]$ to obtain
   \begin{align}
       \| \partial_t u^{\rm non}(t,\cdot)\|_{L^{\mu_1}} &\lesssim \int_0^{t/2} (1+t-\tau)^{-\frac{1}{2}(1-\frac{1}{\mu_1})-1} \||v(\tau,\cdot)|^p\|_{L^1 \cap L^{\mu_1}} d\tau\notag\\
       &\hspace{2cm} + \int_{t/2}^t (1+ t-\tau)^{-1} \| |v(\tau,\cdot)|^p\|_{L^{\mu_1}} d\tau\notag\\
       &\lesssim (1+t)^{-\frac{1}{2}(1-\frac{1}{\mu_1})-1} \|v\|_{X_{12}(T)}^p\int_0^{t/2} (1+\tau)^{-\frac{1}{2}(p-1)} d\tau\notag\\
       &\hspace{2cm}+ (1+t)^{-\frac{1}{2}(p-\frac{1}{\mu_1})} \|v\|_{X_{12}(T)}^p \int_{t/2}^t (1+t-\tau)^{-1} d\tau\notag\\
       &\lesssim (1+t)^{-\frac{1}{2}(1-\frac{1}{\mu_1})-1+[\gamma_1(p)]^+} \|v\|_{X_{12}(T)}^p. \label{Main.Es.1}
   \end{align}
   In the second estimate, we used estimate (\ref{es3.2}) with $\nu = 1, \mu_1$. By using (\ref{es3.3}) and Lemma \ref{LinearEstimates} with $\rho_2=1, \rho_1 =2$, $(d_1, d_2)= (s,0)$ for $\tau \in [0, t/2)$ and $\rho_2=\rho_1 =2$, $(d_1, d_2)= (s,s)$ for $\tau \in [t/2, t]  $ , we arrive at
   \begin{align}
       \|\partial_t u^{\rm non}(t,\cdot)\|_{\dot{H}^s} &\lesssim \int_0^{t/2} (1+t-\tau)^{-\frac{5}{4}-\frac{s}{2}} \||v(\tau,\cdot)|^p\|_{L^1 \cap \dot{H}^s} d\tau\notag\\
       &\hspace{1cm} + \int_{t/2}^t (1+t-\tau)^{-1} \||v(\tau,\cdot)|^p\|_{\dot{H}^s} d\tau\notag\\
       &\lesssim (1+t)^{-\frac{5}{4}-\frac{s}{2}} \|v\|_{X_{12}(T)}^p\int_0^{t/2} (1+\tau)^{-\frac{1}{2}(p-1)} d\tau\notag\\
       &\hspace{1cm} +(1+t)^{-\frac{1}{2}(p-\frac{1}{2})-\frac{s}{2}+\frac{\varepsilon_1}{2}} \|v\|_{X_{12}(T)}^p \int_{t/2}^t(1+t-\tau)^{-1} d\tau\notag\\
       &\lesssim (1+t)^{-\frac{5}{4}-\frac{s}{2}+[\gamma_1(p)]^+}\|v\|_{X_{12}(T)}^p. \label{Main.Es.2}
       \end{align}
   We can easily see that the estimates (\ref{Main.Es.1})-(\ref{Main.Es.2}) lead to (\ref{Esti.Pro3.1.1}).
    
     \textbf{Next, we will prove the estimate (\ref{Esti.Pro3.1.2}).} For $\mu_2 \in \{2,p\}$, using again Lemma \ref{LinearEstimates} with $\rho_2=1,\, \rho_1 =\mu_2$, $(d_1, d_2)=(0,0)$ for $\tau \in [0, t/2)$ and $\rho_2=\rho_1=\mu_2$, $(d_1, d_2)=(0,0)$ for $\tau\in [t/2, t]$ we derive
     \begin{align}
         \|v^{\rm non}(t,\cdot)\|_{L^{\mu_2}} &\lesssim \int_0^{t/2} (1+t-\tau)^{-\frac{1}{2}(1-\frac{1}{\mu_2})} \| |u_t(\tau,\cdot)|^q\|_{L^1 \cap L^{\mu_2}} d\tau\notag\\
         &\hspace{1cm} + \int_{t/2}^t \||u_t(\tau,\cdot)|^q\|_{L^{\mu_2}} d\tau\notag\\
         &\lesssim (1+t)^{-\frac{1}{2}(1-\frac{1}{\mu_2})} \|u_t\|_{X_{11}(T)}^q \int_0^{t/2} (1+\tau)^{-q-\frac{1}{2}(q-1)+q[\gamma_1(p)]^+} d\tau\notag\\
         &\hspace{1cm} +(1+t)^{-q-\frac{1}{2}(q-\frac{1}{\mu_2})+q[\gamma_1(p)]^+} \|u_t\|_{X_{11}(T)}^q \int_{t/2}^t d\tau\notag\\
         &\lesssim (1+t)^{-\frac{1}{2}(1-\frac{1}{\mu_2})} \|u_t\|_{X_{11}(T)}^q. \label{Main.Es.3}
     \end{align}
     To obtain the final estimate, we use the following relation, which is derived directly from condition (\ref{condition1.1}):
     \begin{align*}
        -q -\frac{1}{2}(q-1)+q[\gamma_1(p)]^+ < -1.
     \end{align*}
       Finally, by a similar calculations, we easily obtain
     \begin{align}
         \|v^{\rm non}(t,\cdot)\|_{\dot{H}^s} &\lesssim \int_0^{t/2} (1+t-\tau)^{-\frac{1}{4}-\frac{s}{2}} \| |u_t(\tau,\cdot)|^q\|_{L^1 \cap L^2} d\tau\notag\notag\\
         &\hspace{1cm} +\int_{t/2}^t (1+t-\tau)^{-\frac{s}{2}} \||u_t(\tau,\cdot)|^q\|_{L^2} d\tau\notag\\
         &\lesssim (1+t)^{-\frac{1}{4}-\frac{s}{2}} \|u_t\|_{X_{11}(T)}^q \int_0^{t/2} (1+\tau)^{-q-\frac{1}{2}(q-1)+q[\gamma_1(p)]^+} d\tau\notag\\
         &\hspace{1cm} + (1+t)^{-q-\frac{1}{2}(q-\frac{1}{2})+q[\gamma_1(p)]^+} \|u_t\|_{X_{11}(T)}^q \int_{t/2}^t (1+t-\tau)^{-\frac{s}{2}} d\tau \notag\\
         &\lesssim (1+t)^{-\frac{1}{4}-\frac{s}{2}} \|u_t\|_{X_{11}(T)}^q,\label{Main.Es.4}
     \end{align}
     with $1/2 < s < 1$.
     From the estimates (\ref{Main.Es.1})-(\ref{Main.Es.4}), we immediately conclude (\ref{Esti.Pro3.1.2}). Thus, the proof of Proposition \ref{proposition3.1} is completed.
\end{proof}
 
\begin{proposition}\label{proposition3.2}
    Under the assumptions of Theorem \ref{Theorem1}, the following estimates  hold for all $(u,v)$ and $(\bar{u}, \bar{v}) \in Z_1(T) \times X_{12}(T)$:
    \begin{align}
\|\partial_t \mathcal{N}_{11}[u] - \partial_t \mathcal{N}_{11}[\bar{u}]\|_{Y_{11}(T)} &\lesssim \|v-\bar{v}\|_{Y_{12}(T)} \left(\|v\|_{X_{12}(T)}^{p-1} +\|\bar{v}\|_{X_{12}(T)}^{p-1} \right), \label{Esti.Pro2.2.1}\\
\|\mathcal{N}_{12}[v]-\mathcal{N}_{12}[\bar{v}]\|_{Y_{12}(T)} &\lesssim \|u_t-\bar{u}_t\|_{Y_{11}(T)}\left(\|u_t\|_{X_{11}(T)}^{q-1} + \|\bar{u}_t\|_{X_{11}(T)}^{q-1}\right).\label{Esti.Pro2.2.2}
    \end{align}
\end{proposition}

\begin{proof}
To begin with, we can easily see that
    \begin{align*}
       (\partial_t \mathcal{N}_{11}[u] - \partial_t \mathcal{N}_{11}[\bar{u}])(t,x) &= (\partial_t u^{\rm non} -\partial_t \bar{u}^{\rm non})(t,x) \\
       &= \int_0^t \partial_t \mathcal{K}(t-\tau,x)\ast_x \left(|v(\tau,x)|^p - |\bar{v}(\tau,x)|^p\right) d\tau,\\
       (\mathcal{N}_{12}[v]-\mathcal{N}_{12}[\bar{v}])(t,x) &= (v^{\rm non}- \bar{v}^{\rm non})(t,x) \\
       &= \int_0^t \mathcal{K}(t-\tau,x) \ast_x \left(|u_t(\tau,x)|^q-|\bar{u}_t(\tau,x)|^q\right) d\tau.
    \end{align*}
\textbf{First, we prove the estimate (\ref{Esti.Pro2.2.1})}. Fixing $1< m < \min\{p,q\}$ and using Lemma \ref{LinearEstimates} with $\rho_2 = m,\, \rho_1 = q$, $(d_1, d_2) = (0,0)$ for $\tau \in [0, t/2)$ and $\rho_2= \rho_1 =q$, $(d_1, d_2) = (0,0)$ for $\tau \in [t/2, t]$  to derive
\begin{align*}
    \|\partial_t(u^{\rm non}- \bar{u}^{\rm non})(t,\cdot)\|_{L^q} &\lesssim \int_0^{t/2} (1+t-\tau)^{-\frac{1}{2}(\frac{1}{m}-\frac{1}{q})-1} \left\||v(\tau,\cdot)|^p-|\bar{v}(\tau,\cdot)|^p\right\|_{L^m\cap L^q} d\tau\\
    &\hspace{1cm} +\int_{t/2}^t (1+t-\tau)^{-1} \left\||v(\tau,\cdot)|^p-|\bar{v}(\tau,\cdot)|^p\right\|_{L^q} d\tau.
\end{align*}
After employing H\"older's  inequality, one has
\begin{align*}
    \left\||v(\tau,\cdot)|^p-|\bar{v}(\tau,\cdot)|^p\right\|_{L^{\nu}} \lesssim \|v(\tau,\cdot)-\bar{v}(\tau,\cdot)\|_{L^{\nu p}} \left(\|v(\tau,\cdot)\|_{L^{\nu p}}^{p-1} +\|\bar{v}(\tau,\cdot)\|_{L^{\nu p}}^{p-1}\right),
\end{align*}
for all $\nu \geq 1$. Thanks to Proposition \ref{fractionalGagliardoNirenberg} and the definition of $Y_{12}(T)$, one concludes that
\begin{align*}
    \|v(\tau,\cdot)-\bar{v}(\tau,\cdot)\|_{L^{\nu p}} &\lesssim \|v(\tau,\cdot)-\bar{v}(\tau,\cdot)\|_{L^p}^{1-\theta_3}  \|v(\tau,\cdot)-\bar{v}(\tau,\cdot)\|_{\dot{H}^s}^{\theta_3}\\
    &\lesssim (1+\tau)^{-\frac{1}{2p}(1-\frac{1}{\nu})}\|v- \bar{v}\|_{Y_{12}(T)},
\end{align*}
where
$$\theta_3 = \displaystyle\frac{\frac{1}{p}-\frac{1}{\nu p}}{\frac{1}{p}-\frac{1}{2}+s} \in [0,1]. $$
Moreover, using the estimate (\ref{es3.2}), we can immediately conclude that
\begin{align*}
       \|v(\tau,\cdot)\|_{L^{\nu p}}^{p-1} + \|\bar{v}(\tau,\cdot)\|_{L^{\nu p}}^{p-1} \lesssim (1+\tau)^{-\frac{p-1}{2}(1-\frac{1}{\nu p})} \left(\|v\|_{X_{12}(T)}^{p-1} + \|\bar{v}\|_{X_{12}(T)}^{p-1}\right).
\end{align*}
For this reason, we can conclude
\begin{align*}
    &\left\||v(\tau,\cdot)|^p-|\bar{v}(\tau,\cdot)|^p\right\|_{L^{\nu}} \lesssim (1+\tau)^{-\frac{1}{2p}(1-\frac{1}{\nu})-\frac{p-1}{2}(1-\frac{1}{\nu p})}\|v- \bar{v}\|_{Y_{12}(T)} \left(\|v\|_{X_{12}(T)}^{p-1} + \|\bar{v}\|_{X_{12}(T)}^{p-1}\right).
\end{align*}
From this, for $\nu = m, q$, we arrive at
\begin{align*}
     &\|\partial_t(u^{\rm non}- \bar{u}^{\rm non})(t,\cdot)\|_{L^q}\lesssim (M_1(t) + M_2(t))\|v- \bar{v}\|_{Y_{12}(T)}   \left(\|v\|_{X_{12}(T)}^{p-1} + \|\bar{v}\|_{X_{12}(T)}^{p-1}\right),
\end{align*}
where
\begin{align*}
    M_1(t) &:= (1+t)^{-\frac{1}{2}(\frac{1}{m}-\frac{1}{q})-1}  \int_0^{t/2} (1+\tau)^{-\frac{1}{2p}(1-\frac{1}{m})-\frac{p-1}{2}(1-\frac{1}{mp})} d\tau,\\
    M_2(t) &:= (1+t)^{-\frac{1}{2p}(1-\frac{1}{q})-\frac{p-1}{2}(1-\frac{1}{qp})} \int_{t/2}^t (1+t-\tau)^{-1} d\tau.
\end{align*}
We can easily see that 
\begin{align*}
    M_2(t) &\lesssim (1+t)^{-\frac{1}{2p}(1-\frac{1}{q})-\frac{p-1}{2}(1-\frac{1}{qp})} \log(1+t)\\
    &\lesssim (1+t)^{-1+\frac{1}{2q}-\frac{1}{2p} + [\gamma_1(p)]^+}.
\end{align*}
For the quantity $M_1(t)$, we divide into two cases as follows:
\begin{itemize}[leftmargin=*]
    \item \textbf{Case 1:} If $$-\frac{1}{2p}\left(1-\frac{1}{m}\right)-\frac{p-1}{2}\left(1-\frac{1}{mp}\right) \leq -1, $$ then
    \begin{align*}
        M_1(t) &\lesssim (1+t)^{-\frac{1}{2}(\frac{1}{m}-\frac{1}{q})-1} \log(1+t)\\
        &\lesssim (1+t)^{-1+\frac{1}{2q}-\frac{1}{2p}+[\gamma_1(p)]^+},
    \end{align*}
    where we note that $1 < m < \min\{p, q\}$.
    \item \textbf{Case 2:} If $$-\frac{1}{2p}\left(1-\frac{1}{m}\right)-\frac{p-1}{2}\left(1-\frac{1}{mp}\right) > -1, $$ then
    \begin{align*}
        M_1(t) &\lesssim (1+t)^{-\frac{1}{2}(\frac{1}{m}-\frac{1}{q})-\frac{1}{2p}(1-\frac{1}{m})-\frac{p-1}{2}(1-\frac{1}{mp})} \\
        &\lesssim (1+t)^{-1+\frac{1}{2q}-\frac{1}{2p} +[\gamma_1(p)]^+}.
    \end{align*}
\end{itemize}
 In summary, we have the following estimate:
    \begin{align*}
        &\|\partial_t(u^{\rm non}-\bar{u}^{\rm non})(t,\cdot)\|_{L^q} \lesssim (1+t)^{-1+\frac{1}{2q}-\frac{1}{2p} +[\gamma_1(p)]^+} \|v- \bar{v}\|_{Y_{12}(T)}   \left(\|v\|_{X_{12}(T)}^{p-1} + \|\bar{v}\|_{X_{12}(T)}^{p-1}\right). 
    \end{align*} 
This implies the estimate (\ref{Esti.Pro2.2.1}).\\

\textbf{Next, we prove the estimate (\ref{Esti.Pro2.2.2})}. Specifically, we perform the following steps. Using again Lemma \ref{LinearEstimates} with $1 < m_1 < \min\{p, q\}$ and  $1 < m_2 < \min\{2, p, q\}$, one has
\begin{align*}
    \left\|(v^{\rm non}-\bar{v}^{\rm non})(t,\cdot)\right\|_{L^p} &\lesssim \int_0^t (1+t-\tau)^{-\frac{1}{2}(\frac{1}{m_1}-\frac{1}{p})} \left\||u_t(\tau,\cdot)|^q - |\bar{u}_t(\tau,\cdot)|^q\right\|_{L^{m_1} \cap H_p^{-1}} d\tau\\
    &\lesssim \int_0^t (1+t-\tau)^{-\frac{1}{2}(\frac{1}{m_1}-\frac{1}{p})} \left\||u_t(\tau,\cdot)|^q - |\bar{u}_t(\tau,\cdot)|^q\right\|_{L^{m_1}} d\tau,\\
    \left\|(v^{\rm non}-\bar{v}^{\rm non})(t,\cdot)\right\|_{\dot{H}^{ks}} &\lesssim \int_0^t (1+t-\tau)^{-\frac{1}{2}(\frac{1}{m_2}-\frac{1}{2})-\frac{ks}{2}} \left\||u_t(\tau,\cdot)|^q - |\bar{u}_t(\tau,\cdot)|^q\right\|_{L^{m_2} \cap H^{ks-1}} d\tau\\
    &\lesssim \int_0^t (1+t-\tau)^{-\frac{1}{2}(\frac{1}{m_2}-\frac{1}{2})-\frac{ks}{2}} \left\||u_t(\tau,\cdot)|^q - |\bar{u}_t(\tau,\cdot)|^q\right\|_{L^{m_2}} d\tau,
\end{align*}
where $k=0,1$.
Here, we used to the embedding
\begin{align*}
    \|\psi\|_{L^p} \lesssim \|\psi\|_{H^1_{m_1}} \text{ with } \frac{1}{m_1} \leq \frac{1}{p}+1
\end{align*}
and 
\begin{align*}
    \|\psi\|_{L^2} \lesssim \|\psi\|_{H^{1-ks}_{m_2}} \text{ with } \frac{1}{m_2} \leq \frac{3}{2}-ks. 
\end{align*}
Due to the assumption (\ref{condition1.1.2}), there always exist parameters $m_1$ with $1 < m_1 < \min\{p, q\}$ and $m_2$ with $1 < m_2 < \min\{2, p, q\}$
  satisfying the above conditions. Thanks to H\"older's inequality for $1/m_j = 1/q+ 1/\mu_j $ with $j = 1,2$, one has
  \begin{align*}
  \left\||u_t(\tau,\cdot)|^q - |\bar{u}_t(\tau,\cdot)|^q\right\|_{L^{m_j}} \leq \|u_t(\tau,\cdot)-\bar{u}_t(\tau,\cdot)\|_{L^q} \left(\|u_t(\tau,\cdot)\|_{L^{\mu_j(q-1)}}^{q-1} + \|\bar{u}_t(\tau,\cdot)\|_{L^{\mu_j(q-1)}}^{q-1}\right)
  \end{align*}
From the definition of the space $Y_{11}(T)$, we see that
\begin{align*}
    \|u_t(\tau,\cdot)-\bar{u}_t(\tau,\cdot)\|_{L^q}  \lesssim (1+\tau)^{-1+\frac{1}{2q}-\frac{1}{2p}+ [\gamma_1(p)]^+} \|u_t-\bar{u}_t\|_{Y_{11}(T)}.
\end{align*}
On the other hand, using Proposition \ref{fractionalGagliardoNirenberg}, we also obtain
\begin{align*}
    &\|u_t(\tau,\cdot)\|_{L^{\mu_j(q-1)}}^{q-1} + \|\bar{u}_t(\tau,\cdot)\|_{L^{\mu_j(q-1)}}^{q-1} \\
    &\hspace{2cm}\lesssim (1+\tau)^{-(q-1)-\frac{1}{2}(q-1-\frac{1}{\mu_j})+(q-1)[\gamma_1(p)]^+} \left(\|u_t\|_{X_{11}(T)}^{q-1} + \|\bar{u}_t\|_{X_{11}(T)}^{q-1}\right),
\end{align*}
provided that it is satisfied:
\begin{align*}
    \displaystyle\frac{\displaystyle\frac{1}{q}-\frac{1}{\mu_j(q-1)}}{\displaystyle\frac{1}{q}-\frac{1}{2}+s} \in [0, 1].
\end{align*}
This implies
\begin{align*}
    &\left\||u_t(\tau,\cdot)|^q - |\bar{u}_t(\tau,\cdot)|^q\right\|_{L^{m_j}}\\ &\hspace{1cm}\lesssim (1+\tau)^{-(q-1)-\frac{1}{2}(q+1+\frac{1}{p}-\frac{1}{m_j})+q[\gamma_1(p)]^+} \|u_t-\bar{u}_t\|_{Y_{11}(T)}\left(\|u_t\|_{X_{11}(T)}^{q-1} + \|\bar{u}_t\|_{X_{11}(T)}^{q-1}\right),
\end{align*}
with $j = 1,2$. From this, we obtain
\begin{align*}
    \|(v^{\rm non} -\bar{v}^{\rm non})(t,\cdot)\|_{L^p} &\lesssim (M_3(t) + M_4(t)) \|u_t-\bar{u}_t\|_{Y_{11}(T)}\left(\|u_t\|_{X_{11}(T)}^{q-1} + \|\bar{u}_t\|_{X_{11}(T)}^{q-1}\right),\\
    \|(v^{\rm non}-\bar{v}^{\rm non})(t,\cdot)\|_{\dot{H}^s} &\lesssim (M_5(t) + M_6(t)) \|u_t-\bar{u}_t\|_{Y_{11}(T)}\left(\|u_t\|_{X_{11}(T)}^{q-1} + \|\bar{u}_t\|_{X_{11}(T)}^{q-1}\right),
\end{align*}
where
\begin{align*}
    M_3(t) &:=  (1+t)^{-\frac{1}{2}(\frac{1}{m_1}-\frac{1}{p})}\int_0^{t/2} (1+\tau)^{-(q-1)-\frac{1}{2}(q+1+\frac{1}{p}-\frac{1}{m_1})+ q[\gamma_1(p)]^+} d\tau,\\
    M_4(t) &:= (1+t)^{-(q-1)-\frac{1}{2}(q+1+\frac{1}{p}-\frac{1}{m_1})+ q[\gamma_1(p)]^+} \int_{t/2}^t (1+t)^{-\frac{1}{2}(\frac{1}{m_1}-\frac{1}{p})} d\tau,\\
    M_5(t) &:= (1+t)^{-\frac{1}{2}(\frac{1}{m_2}-\frac{1}{2})-\frac{ks}{2}} \int_0^{t/2} (1+\tau)^{-(q-1)-\frac{1}{2}(q+1+\frac{1}{p}-\frac{1}{m_2})+ q[\gamma_1(p)]^+} d\tau,\\
    M_6(t) &:= (1+t)^{-(q-1)-\frac{1}{2}(q+1+\frac{1}{p}-\frac{1}{m_2})+ q[\gamma_1(p)]^+} \int_{t/2}^t (1+t-\tau)^{-\frac{1}{2}(\frac{1}{m_2}-\frac{1}{2})-\frac{ks}{2}} d\tau.
\end{align*}
For the integral $M_3(t)$, we consider two cases as follows:
\begin{itemize}[leftmargin=*]
    \item \textbf{Case 1:} If
    $$ -(q-1)-\frac{1}{2}\left(q+1+\frac{1}{p}-\frac{1}{m_1}\right)+ q[\gamma_1(p)]^+ \leq -1, $$
    we  immediately conclude $M_3(t) \lesssim 1$ for all $t > 0$, due to $m_1 < \min\{p, q\}$.
    \item \textbf{Case 2:} If
    $$ -(q-1)-\frac{1}{2}\left(q+1+\frac{1}{p}-\frac{1}{m_1}\right)+ q[\gamma_1(p)]^+ > -1, $$
    one has
    \begin{align*}
        M_3(t) \lesssim (1+t)^{1-q-\frac{1}{2}(q-1)+q[\gamma_1(p)]^+} \lesssim 1,
    \end{align*}
    where we note that the condition (\ref{condition1.1}) implies
    \begin{align*}
        -q-\frac{1}{2}(q-1)+ q[\gamma_1(p)]^+ < -1.
    \end{align*}
    
\end{itemize}
Therefore, we have $M_3(t) \lesssim 1$ for all $t > 0$. Similarly, we see that $$-\frac{1}{2}\left(\frac{1}{m_1}-\frac{1}{p}\right) > -1.$$ This leads to $M_4(t) \lesssim 1$. For these reasons, we can conclude that
    \begin{align}
        \|(v^{\rm non} -\bar{v}^{\rm non})(t,\cdot)\|_{L^p} \lesssim \|u_t-\bar{u}_t\|_{Y_{11}(T)}\left(\|u_t\|_{X_{11}(T)}^{q-1} + \|\bar{u}_t\|_{X_{11}(T)}^{q-1}\right). \label{Main.Es.6}
    \end{align}
For the integral $M_5(t)$, we consider two cases as follows:
\begin{itemize}[leftmargin=*]
    \item \textbf{Case 1:} If
    $$-(q-1)-\frac{1}{2}\left(q+1+\frac{1}{p}-\frac{1}{m_2}\right)+ q[\gamma_1(p)]^+ \leq -1, $$
    we immediately conclude
    \begin{align*}
        M_5(t) \lesssim (1+t)^{-\frac{1}{2}(\frac{1}{m_2}-\frac{1}{2})-\frac{ks}{2}} \lesssim (1+t)^{-\frac{1}{2}(\frac{1}{p}-\frac{1}{2}+ks)},
    \end{align*}
    due to $m_2 < \min\{2, p,q\}$.
    \vspace{0.5cm}
    \item \textbf{Case 2:} If 
    $$-(q-1)-\frac{1}{2}\left(q+1+\frac{1}{p}-\frac{1}{m_2}\right)+ q[\gamma_1(p)]^+ > -1, $$
    we see that
    \begin{align*}
        M_5(t) &\lesssim (1+t)^{-\frac{1}{2}(\frac{1}{p}-\frac{1}{2}+ks)+\left(1-q-\frac{1}{2}(q-1)+q[\gamma_1(p)]^+\right)}\\
        &\lesssim (1+t)^{-\frac{1}{2}(\frac{1}{p}-\frac{1}{2}+ks)}.
    \end{align*}
\end{itemize}
Similarly, we also conclude $M_6(t) \lesssim (1+t)^{-\frac{1}{2}(\frac{1}{p}-\frac{1}{2}+ks)}$ due to $$-\frac{1}{2}\left(\frac{1}{m_2}-\frac{1}{2}\right)-\frac{ks}{2} > -1.$$ This leads to
\begin{align}\label{Main.Es.7}
    \|(v^{\rm non}-\bar{v}^{\rm non})(t,\cdot)\|_{\dot{H}^{ks}}\lesssim (1+t)^{-\frac{1}{2}(\frac{1}{p}-\frac{1}{2}+ks)} \|u_t-\bar{u}_t\|_{Y_{11}(T)}\left(\|u_t\|_{X_{11}(T)}^{q-1} + \|\bar{u}_t\|_{X_{11}(T)}^{q-1}\right),
\end{align}
for $k = 0,1$.
The estimates (\ref{Main.Es.6})-(\ref{Main.Es.7}) imply (\ref{Esti.Pro2.2.2}). Thus, we immediately conclude Proposition \ref{proposition3.2}.
\end{proof}
\textbf{Proof of Theorem \ref{Theorem1}.}
Using Lemma \ref{LinearEstimates} and the definitions of $u^{\rm lin}, \, v^{\rm lin}$, we immediately have
\begin{align*}
    \|(\partial_t u^{\rm lin}, v^{\rm lin})\|_{X_1(T)} \lesssim \varepsilon\|(u_0, u_1, v_0, v_1)\|_{\mathcal{D}_{q,p}^1(s)}.
\end{align*}
For this reason, combining with Propositions \ref{proposition3.1} and \ref{proposition3.2}, we obtain the following two important estimates for all $(u,v)$ and $(\bar{u}, \bar{v}) \in Z_1(T) \times X_{12}(T)$:
    \begin{align}
        &\|(\partial_t \mathcal{N}_{11}[u],\, \mathcal{N}_{12}[v])\|_{X_1(T)} \leq C_1 \varepsilon \|(u_0, u_1)\|_{\mathcal{D}_{q,p}^1(s)} + C_1\|(u_t, v)\|_{X_1(T)}^p + C_1\|(u_t, v)\|_{X_1(T)}^q, \label{Es.Pro2.1}
        \end{align}
        and
        \begin{align}
        &\|(\partial_t \mathcal{N}_{11}[u],\, \mathcal{N}_{12}[v])- (\partial_t \mathcal{N}_{11}[\bar{u}], \mathcal{N}_{12}[\bar{v}])\|_{Y_1(T)}\notag\\ 
        &\quad\leq C_2 \|(u_t, v)-(\bar{u}_t, \bar{v})\|_{Y_1(T)}
        \notag\\ 
        &\hspace{2cm} \times \big(\|(u_t,v)\|_{X_1(T)}^{p-1} + \|(u_t, v)\|_{X_1(T)}^{q-1}+ \|(\bar{u}_t,\bar{v})\|_{X_1(T)}^{p-1} + \|(\bar{u}_t, \bar{v})\|_{X_1(T)}^{q-1}\big). \label{Es.Pro2.2}
    \end{align}
    Let us fix 
    \begin{align}\label{Cons1}
    M := 2C_1 \|(u_0, u_1)\|_{\mathcal{D}_{q,p}^1(s)}
    \end{align}
    and the parameter $\varepsilon_0$ satisfies
    \begin{align*}
        \max\{C_1, 2C_2\}\left(M^{p-1}\varepsilon_0^{p-1}+ M^{q-1} \varepsilon_0^{q-1}\right) \leq \frac{1}{2}.
    \end{align*}
   Next, taking the recurrence sequences $\{u_j\}_{j=-1}^\infty \subset Z_1(T)$ with $u_{-1} = 0;\,  u_{j} = \mathcal{N}_{11}[u_{j-1}]$ and $\{v_j\}_{j=-1}^\infty \subset X_{12}(T)$ with $v_{-1} = 0,\, v_j =\mathcal{N}_{12}[v_{j-1}]$ for $j = 0,1,2,...$, we employ the estimate (\ref{Es.Pro2.1}) to conclude that $$  \|(\partial_t u_j, v_j)\|_{X_1(T)} \leq M\varepsilon,$$ for all $j \in \mathbb{N}$ and $\varepsilon \in (0, \varepsilon_0]$. 
Moreover, using estimate (\ref{Es.Pro2.2}), it is a fact that 
\begin{align*}
    \|(\partial_t u_{j+1}, v_{j+1})-(\partial_t u_j, v_j)\|_{Y_1(T)}
    &\leq \frac{1}{2} \|(\partial_t u_{j}, v_j)-(\partial_t u_{j-1}, v_{j-1})\|_{Y_1(T)},
\end{align*}
so that $$\{(\partial_t u_j, v_j)\}_{j=-1}^{\infty} \subset X_1(T, M\varepsilon) \subset X_{11}(T, M\varepsilon) \times X_{12}(T, M\varepsilon)$$ is a Cauchy sequence in the Banach
space $Y_1(T)$. Furthermore, using Lemma \ref{lemma1.2}, there exists a function $(\varphi^*, v^*) \in X_{11}(T, M\varepsilon) \times X_{12}(T, M\varepsilon)$ satisfies
\begin{align*}
     (\partial_t u_j, v_j) \to (\varphi^*, v^*)   \text{ in } X_{11}(T, M\varepsilon) \times X_{12}(T, M\varepsilon),
\end{align*}
as $j \to \infty$, with the  metric $Y_1(T)$.
Next, we consider the following function:
\begin{align*}
    u^{*}(t,x) = \int_0^t \varphi^*(\tau,x) d\tau + u_0(x).
\end{align*}
Additionally, one has
\begin{align*}
    \lim_{t \to t_0} \|u^*(t,\cdot)-u^*(t_0,\cdot)\|_{L^2} \lesssim \lim_{t \to t_0} \left|\int_{t_0}^t \|\varphi^*(\tau,\cdot)\|_{X_{11}(T)} d\tau\right| \leq M\varepsilon\lim_{t \to t_0}  |t-t_0| = 0,
\end{align*}
for all $t_0 \in [0,T]$. From this, we obtain
\begin{align*}
    u^* \in \mathcal{C}([0, T], L^2).
\end{align*}
Moreover, thanks to again the estimate (\ref{Es.Pro2.2}), it holds 
\begin{align*}
    &\|(\partial_t u_{j+1}, v_{j+1}) - (\partial_t \mathcal{N}_{11}[u^*], \mathcal{N}_{12}[v^*])\|_{Y_1(T)} \\
    &\quad\leq C_2\|(\partial_t u_j, v_j) - (\varphi^*, v^*)\|_{Y_1(T)}\\
    &\qquad\qquad\times\left(\|(\partial_t u_j, v_j)\|_{X_1(T)}^{p-1}+\|(\varphi^*, v^*)\|_{X_1(T)}^{p-1} + \|(\partial_t u_j, v_j)\|_{X_1(T)}^{q-1}+\|(\varphi^*, v^*)\|_{X_1(T)}^{q-1}\right)\\
    &\quad\leq \frac{1}{2} \|(\partial_t u_j, v_j) - (\varphi^*, v^*)\|_{Y_1(T)},
\end{align*}
that is, 
\begin{align*}
    (\partial_t u_j, v_j) \to (\partial_t \mathcal{N}_{11}[u^*], \mathcal{N}_{12}[v^*]) \quad\text{ as } \quad j \to \infty.
\end{align*}
Therefore, it can immediately yield
\begin{align*}
    (\varphi^*, v^*) \equiv (\partial_t\mathcal{N}_{11}[u^*], \mathcal{N}_{12}[v^*]),
\end{align*}
that is, there exists a function $\gamma(x)$ satisfies
\begin{align}\label{Eq1}
    u^*(t,x) = \mathcal{N}_{11}[u^*](t,x) + \gamma(x),
\end{align}
for all $t \geq 0, x 
\in \mathbb{R}^n$.
 Taking $t=0$ in (\ref{Eq1}), we can see that $\gamma(x) \equiv 0$ on $\mathbb{R}^n$ due to $u^*(0,x) = u_0(x)$. For this reason, we have a solution $$(u^*, v^*) = (\mathcal{N}_{11}[u^*], \mathcal{N}_{12}[v^*]) \quad\text{ in }\quad \mathcal{C}([0,T], L^2) \times X_{12}(T, M\varepsilon), $$ for all $T >0$. Since $T$ is arbitrary, the solution is global, that is, $$(u^*, v^*) \in \mathcal{C}([0, \infty), L^2) \times X_{12}(\infty, M\varepsilon).$$
 Next, we need to show that $(u^*, v^*) \in \mathcal{C}^1([0,\infty), H^s \cap L^{q}) \times \mathcal{C}([0, \infty), H^s \cap L^p)$, that is, $(\varphi^*, v^*) \in \mathcal{C}([0,\infty), H^s \cap L^{q}) \times \mathcal{C}([0, \infty), H^s \cap L^p) $. To prove this property, we recall that
 \begin{align*}
     \varphi^*(t,x) = \partial_t \mathcal{N}_{11}[u^*](t,x) &= \partial_t u^{\rm lin}(t,x) + \partial_t \int_0^t \mathcal{K}(t-\tau,x) \ast_x |v^*(\tau,x)|^p d\tau\\
     &= \partial_t u^{\rm lin}(t,x) +  \int_0^t \partial_t \mathcal{K}(t-\tau,x) \ast_x |v^*(\tau,x)|^p d\tau,\\
     v^*(t,x) =  \mathcal{N}_{12}[v^*](t,x) &=  v^{\rm lin}(t,x) +  \int_0^t \mathcal{K}(t-\tau,x) \ast_x |\partial_t u^*(\tau,x)|^q d\tau\\
     &= v^{\rm lin}(t,x) +  \int_0^t \mathcal{K}(t-\tau,x) \ast_x |\varphi^*(\tau,x)|^q d\tau.
 \end{align*}
 Since the linear parts obviously satisfy continuity, it suffices to show that
\begin{align}
    \int_0^t \partial_t \mathcal{K}(t-\tau,x) \ast_x |v^*(\tau,x)|^p d\tau \in \mathcal{C}([0,\infty), H^s \cap L^{q}), \label{EQ2}\\
    \int_0^t  \mathcal{K}(t-\tau,x) \ast_x |\varphi^*(\tau,x)|^q d\tau \in \mathcal{C}([0,\infty), H^s \cap L^{p}) .
    \label{EQ3}
\end{align}
Using again Lemmas \ref{LinearEstimates} and \ref{Lemma2.3}, we have the following estimates:
\begin{align*}
    \| \mathcal{K}(t-\tau,x) \ast_x |\varphi^*(\tau,x)|^q\|_{L^{p}} &\lesssim (1+\tau)^{-q-\frac{n}{2}(q-\frac{1}{p})+q[\gamma_1(p)]^+} \|\varphi^*\|_{X_{11}(\infty)}^q,\\
    \| \mathcal{K}(t-\tau,x) \ast_x |\varphi^*(\tau,x)|^q\|_{\dot{H}^{ks}} &\lesssim (1+\tau)^{-q-\frac{n}{2}(q-\frac{1}{2})-\frac{ks}{2}} \|\varphi^*\|_{X_{11}(\infty)}^q
    \end{align*}
    and
    \begin{align*}
    \|\partial_t \mathcal{K}(t-\tau,x) \ast_x |v^*(\tau,x)|^p\|_{L^q} &\lesssim (1+\tau)^{-\frac{1}{2}(p-\frac{1}{2})} \|v^*\|_{X_{12}(\infty)}^p,\\
     \|\partial_t \mathcal{K}(t-\tau,x) \ast_x |v^*(\tau,x)|^p\|_{\dot{H}^{ks}} &\lesssim (1+\tau)^{-\frac{1}{2}(p-\frac{1}{2})-\frac{ks}{2}+\varepsilon_1} \|v^*\|_{X_{12}(\infty)}^p
\end{align*}
where $k=0,1$ and $\varepsilon_1$ is an arbitrarily small positive constant. Therefore, the Lebesgue convergence theorem in the Bochner integral immediately implies (\ref{EQ2}) and (\ref{EQ3}).

 Finally, we establish the uniqueness of the global solution 
$$(u^*, v^*) \in \left(\mathcal{C}([0, \infty), L^2) \cap \mathcal{C}^1([0, \infty), H^s \cap L^q)\right) \times \mathcal{C}([0, \infty), H^s \cap L^p).$$ Let $(u^*, v^*)$ and $(\bar{u}^*, \bar{v}^*)$ are solutions of (\ref{Main.Eq.1}) belonging to this space. One can easily see that there exists a constant $M(T)$ such that $$\|(\partial_t u^*, v^*)\|_{X_1(T)}^{p-1}+ \|(\partial_t u^*, v^*)\|_{X_1(T)}^{q-1} + \|(\partial_t \bar{u}^*, \bar{v}^*)\|_{X_1(T)}^{p-1}+ \|(\partial_t \bar{u}^*, \bar{v}^*)\|_{X_1(T)}^{q-1} \leq M(T).$$ Performing the same proof steps of Proposition \ref{proposition3.2} we obtain
 \begin{align*}
      \|(\partial_t u^*, v^*)-(\partial_t \bar{u}^*, \bar{v}^*)\|_{Y_1(t)} \lesssim M(T) \int_0^t \|(\partial_t u^*, v^*)-(\partial_t \bar{u}^*, \bar{v}^*)\|_{Y_1(\tau)}  d\tau,
 \end{align*}
 for all $t \in [0, T]$.
From this, the
Gronwall inequality implies $ (\partial_t u^*, v^*) \equiv (\partial_t \bar{u}^*, \bar{v}^*)$ on $[0,T]$. Because $T $ is an arbitrary positive number, $(\partial_t u^*, v^*) \equiv (\partial_t \bar{u}^*, \bar{v}^*)$ on $[0, \infty)$. We note that $u^*(0,x) = \bar{u}^*(0,x) = u_0(x)$, this implies $(u^*, v^*) \equiv (\bar{u}^*, \bar{v}^*)$ on $[0, \infty)$.
 Hence, the proof of Theorem \ref{Theorem1} is completed.
   
\subsection{Proof of Theorem \ref{Theorem3}}
To begin with, the result related to the linear estimates in the two-dimensional space is restated in the following lemma. 
\begin{lemma}[Linear estimates in 2D]\label{LinearEstimates_2D}
    Let $n=2$, $1 \leq \rho_2 \leq \rho_1 < \infty$ with $\rho_1 \ne 1$ and $0 \leq d_1 \leq d_2$.  Then, the following estimates hold for all $t > 0$ (see Theorem 1.1 in \cite{Ikeda2019}):
    \begin{align*}
        &\| |\nabla|^{d_1} \mathcal{K}(t,x)\ast_x \varphi(x)\|_{L^{\rho_1}} \\
        &\qquad\lesssim (1+t)^{-(\frac{1}{\rho_2}-\frac{1}{\rho_1})-\frac{d_1-d_2}{2}} \| |\nabla|^{d_2} \chi_{\rm L}(|\nabla|) \varphi\|_{L^{\rho_2}} + e^{-ct} \| |\nabla|^{d_1} \chi_{\rm H}(|\nabla|) \varphi\|_{H_{\rho_1}^{\sigma(\rho_1)-1}},\\ 
        &\| |\nabla|^{d_1} \partial_t \mathcal{K}(t,x)\ast_x \varphi(x)\|_{L^{\rho_1}}\\
        &\qquad\lesssim (1+t)^{-(\frac{1}{\rho_2}-\frac{1}{\rho_1})-\frac{d_1-d_2}{2}-1} \| |\nabla|^{d_2} \chi_{\rm L}(|\nabla|) \varphi\|_{L^{\rho_2}} + e^{-ct} \| |\nabla|^{d_1} \chi_{\rm H}(|\nabla|) \varphi\|_{H_{\rho_1}^{\sigma(\rho_1)}},
    \end{align*}
    where $c$ is a suitable positive constant. Moreover, we also obtain the following estimates (see Propositions 4.1 and 4.2 in \cite{DabbiccoEbert2017}):
    \begin{align*}
        \||\nabla|^{d_1} \partial_t^2 \mathcal{K}(t,x) \ast_{x} \varphi(x) \|_{L^{\rho_1}} \lesssim (1+t)^{-(\frac{1}{\rho_2}-\frac{1}{\rho_1})-\frac{d_1}{2}-2} \|\chi_{\rm L}(|\nabla|)\varphi\|_{L^{\rho_2}} + e^{-ct} \|\chi_{\rm H}(|\nabla|)\varphi\|_{H_{\rho_1}^{\delta(\rho_1)+d_1+1}},  
    \end{align*}
    where
    \begin{align*}
    \sigma(\rho_1) := \left|\frac{1}{\rho_1}-\frac{1}{2}\right| \text{ and } \quad \delta(\rho_1) := 2\left|\frac{1}{\rho_1}-\frac{1}{2}\right|.
\end{align*}
\end{lemma}
\begin{remark}
    \fontshape{n}
\selectfont
As in Lemma \ref{LinearEstimates}, the third estimate in Lemma \ref{LinearEstimates_2D} is used solely to determine the data space required to obtain the result in Theorem \ref{Theorem3}. The first and second estimates will be used frequently in our computations.
\end{remark}

Under the assumptions of Theorem \ref{Theorem3},  we define the following function spaces:
\begin{align*}
    X_{21}(T) :=   L^\infty([0, T], H^{s} \cap L^\alpha), \quad X_{22}(T) := L^\infty([0,T], H^s \cap L^p)
\end{align*}
and 
\begin{align*}
    X_2(T) := X_{21}(T) \times X_{22}(T)
\end{align*}
with the norms
\begin{align*}
    \|\varphi_1\|_{X_{21}(T)} := &\sup _{t \in [0,T]} \bigg\{(1+t)^{2-\frac{1}{\alpha}-[\gamma_2(p)]^+}\|\varphi_1(t,\cdot)\|_{L^\alpha}\\
    &\quad\quad+ (1+t)^{\frac{3}{2}-[\gamma_2(p)]^+} \|\varphi_1(t,\cdot)\|_{L^2}+ (1+t)^{\frac{3}{2}+\frac{s}{2}-[\gamma_2(p)]^+}\|\varphi_1(t,\cdot)\|_{\dot{H}^{s}}\bigg\},\\
    \|\varphi_2\|_{X_{22}(T)} := &\sup_{t \in [0,T]} \left\{(1+t)^{1-\frac{1}{p}}\|\varphi_2(t, \cdot)\|_{L^p} + (1+t)^{\frac{1}{2}} \|\varphi_2(t,\cdot)\|_{L^2} +(1+t)^{\frac{1}{2}+\frac{s}{2}}  \|\varphi_2(t,\cdot)\|_{\dot{H}^{s}}\right\}
\end{align*}
and
\begin{align*}
    \|(\varphi_1, \varphi_2)\|_{X_2(T)}:= \|\varphi_1\|_{X_{21}(T) } +\|\varphi_2\|_{X_{22}(T)},
\end{align*}
respectively. Moreover, for $\beta := \max\{2,p\}$ and $\kappa$ is an arbitrarily small positive constant, we also define the spaces
\begin{align*}
    Y_{21}(T) := L^\infty([0, T], L^2), \quad  Y_{22}(T) := L^\infty([0, T], H^{1-\kappa} \cap L^\beta)
\end{align*}
and
\begin{align*}
    Y_2(T) := Y_{21}(T) \times Y_{22}(T),
\end{align*}
with the norms
\begin{align*}
    \|\varphi_1\|_{Y_{21}(T)} &:= \sup_{t \in [0,T]} \bigg\{(1+t)^{\frac{1}{2}+\frac{1}{\beta}-[\gamma_2(p)]^+}\|\varphi_1(t,\cdot)\|_{L^2}\bigg\}\\
    \|\varphi_2\|_{Y_{22}(T)} &:= \sup_{t \in [0,T]}\bigg\{\|\varphi_2(t, \cdot)\|_{L^\beta} + (1+t)^{\frac{1}{\beta}-\frac{1}{2}} \|\varphi_2(t,\cdot)\|_{L^2}+ (1+t)^{\frac{1}{\beta}-\frac{1}{2} + \frac{1-\kappa}{2}}\|\varphi_2(t,\cdot)\|_{\dot{H}^{1-\kappa}}\bigg\}
\end{align*}
and
\begin{align*}
    \|(\varphi_1, \varphi_2)\|_{Y_2(T)} := \|\varphi_1\|_{Y_{21}(T)} + \|\varphi_2\|_{Y_{22}(T)},
\end{align*}
respectively. In addition, we note that
\begin{align*}
    Z_2(T) := \{\varphi \text{ is a first-order differentiable (in time) function on $[0,T]$ and } \varphi_t \in X_{21}(T)\}
\end{align*}
and
\begin{align*}
X_{21}(T, M) &:= \big\{\varphi_k \in X_{21}(T), \|\varphi_1\|_{X_{21}(T)} \leq M\big\},\\
    X_{22}(T, M) &:= \big\{\varphi_k \in X_{22}(T), \|\varphi_2\|_{X_{22}(T)} \leq M\big\},\\
    X_2(T,M) &:= \big\{(\varphi_1, \varphi_2) \in X_2(T), \|(\varphi_1, \varphi_2)\|_{X_2(T)} \leq M \},
\end{align*}
for all $T, M > 0$. We can easily see that $X_{21}(T) \subset Y_{21}(T)$ and $X_{22}(T) \subset Y_{22}(T)$ thanks to Proposition \ref{fractionalGagliardoNirenberg} and $s > 1$.
Similarly to Lemma \ref{lemma1.2}, we have the following important property.
\begin{lemma}\label{lemma2.5}
   According to the assumptions of Theorem \ref{Theorem3}, $X_{21}(T, M)$  is a closed subset of $Y_{21}(T)$ with respect to the metric $Y_{21}(T)$. Moreover, $X_{22}(T, M)$  is also a closed subset of $Y_{22}(T)$ with respect to the metric $Y_{22}(T)$. 
\end{lemma}
Next, we consider the operators $\mathcal{N}_{21}$ and $\mathcal{N}_{22}$ on the spaces $Z_2(T)$ and $X_{22}(T)$, respectively, as follows:
\begin{align*}
    \mathcal{N}_{21}: u \in Z_2(T) \to \mathcal{N}_{21}[u] := u^{\rm lin} + u^{\rm non}, \quad  \mathcal{N}_{22}: v\in X_{22}(T) \to \mathcal{N}_{22}[v] := v^{\rm lin} + v^{\rm non}.
\end{align*}
In order to prove Theorem \ref{Theorem3}, we will show a pair $(u^*, v^*)$ that satisfies the following integral equation
\begin{align*}
    (u^*, v^*) = (\mathcal{N}_{21}[u^*], \mathcal{N}_{22}[v^*]),
\end{align*}
for all $T > 0$.
Specifically, we are going to proof the following lemmas and propositions.
\begin{lemma}\label{lemma2.4}
   Under the assumptions of Theorem \ref{Theorem3}, the following estimates hold for all $(u,v) \in Z_2(T) \times X_{22}(T)$: 
   \begin{align}
       \||v(\tau,\cdot)|^p\|_{L^{\nu}} &\lesssim (1+\tau)^{-p +\frac{1}{\nu}} \|v\|_{X_{22}(T)}^p,\label{Es.of.Lemma2.4.2}\\
       \||v(\tau,\cdot)|^p\|_{\dot{H}_{q}^{\sigma(q)}} &\lesssim (1+\tau)^{-p+\frac{1}{q}-\frac{\sigma(q)}{2}} \|v\|_{X_{22}(T)}^p \quad\text{ with } q \in (1,2), \label{Es.of.Lemma2.4.3}\\
       \||v(\tau,\cdot)|^p\|_{\dot{H}^{s}} &\lesssim (1+\tau)^{-p+\frac{1}{2}-\frac{s}{2}+\varepsilon_2} \|v\|_{X_{22}(T)}^p, \label{Es.of.Lemma2.4.7} 
   \end{align}
where $\nu \geq 1$, $\sigma(q) := 1/q-1/2$ and $\varepsilon_2$ is an arbitrarily small positive constant. Moreover, we also obtain the following estimates:
\begin{align}
    \||u_t(\tau,\cdot)|^q\|_{L^{\nu}} &\lesssim (1+\tau)^{-2q+\frac{1}{\nu}+q [\gamma_2(p)]^+} \|u_t\|_{X_{21}(T)}^q, \label{Es.of.Lemma2.4.5}\\
    \||u_t(\tau,\cdot)|^q\|_{\dot{H}^{s-1}} &\lesssim (1+\tau)^{-2q+\frac{1}{2}  -\frac{s-1}{2}+ q[\gamma_2(p)]^+} \|u_t\|_{X_{21}(T)}^q. \label{Es.of.Lemma2.4.6}
\end{align}
\end{lemma}

\begin{proof}
First, thanks to Proposition \ref{fractionalGagliardoNirenberg}, one has
\begin{align*}
    \||v(\tau,\cdot)|^p\|_{L^\nu} = \|v(\tau,\cdot)\|_{L^{p\nu}}^p &\lesssim \|v(\tau,\cdot)\|_{L^p}^{p(1-\omega_1)} \|v(\tau,\cdot)\|_{\dot{H}^{s}}^{p\omega_1}\lesssim (1+\tau)^{-p+\frac{1}{\nu}} \|v\|_{X_{22}(T)}^p,
    \end{align*}
    where
\begin{align*}
    \omega_1 = \frac{\frac{1}{p}-\frac{1}{p\nu}}{\frac{1}{p}-\frac{1}{2}+\frac{s}{2}}\in [0, 1].
\end{align*}
   Next, we can easily calculate
    \begin{align*}
    \||u_t(\tau,\cdot)|^q\|_{L^\nu} = \|u_t(\tau,\cdot)\|_{L^{q\nu}}^q &\lesssim \|u_t(\tau,\cdot)\|_{L^\alpha}^{q(1-\omega_2)} \|u_t(\tau,\cdot)\|_{\dot{H}^{s}}^{q \omega_2}\\
    &\lesssim (1+\tau)^{-2q+\frac{1}{\nu}+q[\gamma_2(p)]^+} \|u_t\|_{X_{21}(T)}^q,
\end{align*}
where
\begin{align*}
    \omega_2 = \frac{\frac{1}{\alpha}-\frac{1}{q\nu}}{\frac{1}{\alpha}-\frac{1}{2}+\frac{s}{2}} \in [0, 1].
\end{align*}
These lead to the estimates (\ref{Es.of.Lemma2.4.2}) and (\ref{Es.of.Lemma2.4.5}).
Now, we are going to prove (\ref{Es.of.Lemma2.4.7}). Observing that $1 < s < \min\{2,p\}$ and fixing $\varepsilon_2' = 2\varepsilon_2/(p-1)$, one employs Proposition \ref{FractionalPowers} and \ref{Embedding} to derive
\begin{align*}
    \||v(\tau,\cdot)|^p\|_{\dot{H}^{s}} &\lesssim \|v(\tau,\cdot)\|_{\dot{H}^{s}} \left(\|v(\tau,\cdot)\|_{\dot{H}^{1-\varepsilon_2'}} + \|v(\tau,\cdot)\|_{\dot{H}^{s}}\right)^{p-1}\\
    &\lesssim (1+\tau)^{-\frac{1}{2}-\frac{s}{2}-\frac{p-1}{2}-\frac{(p-1)(1-\varepsilon_2')}{2}}\|v\|_{X_{22}(T)}^p\\
    &\lesssim (1+\tau)^{-p+\frac{1}{2}-\frac{s}{2}+\varepsilon_2} \|v\|_{X_{22}(T)}^p.
    \end{align*}
    Next, we will prove the estimates (\ref{Es.of.Lemma2.4.3}). Applying Proposition \ref{chainrule} with $1/q = (p-1)/r_1 + 1/r_2$, we arrive at
    \begin{align*}
        \||v(\tau,\cdot)|^p\|_{\dot{H}_{q}^{\sigma(q)}} &\lesssim \|v(\tau,\cdot)\|_{L^{r_1}}^{p-1} \|v(\tau,\cdot)\|_{\dot{H}^{\sigma(q)}_{r_2}}.
    \end{align*}
    Using again Proposition \ref{fractionalGagliardoNirenberg}, we obtain
    \begin{align*}
        \|v(\tau,\cdot)\|_{L^{r_1}} &\lesssim \|v(\tau,\cdot)\|_{L^2}^{1-\omega_3}\|v(\tau,\cdot)\|_{\dot{H}^{s}}^{\omega_3} \lesssim (1+\tau)^{-1+\frac{1}{r_1}} \|v\|_{X_{22}(T)},\\
        \|v(\tau,\cdot)\|_{\dot{H}^{\sigma(q)}_{r_2}} &\lesssim \|v(\tau,\cdot)\|_{L^2}^{1-\omega_4} \|v(\tau,\cdot)\|_{\dot{H}^{s}}^{\omega_4} \lesssim (1+\tau)^{-1+\frac{1}{r_2}-\frac{\sigma(q)}{2}}\|v\|_{X_{22}(T)}.
    \end{align*}
    Combining the three estimates above, we derive the estimates (\ref{Es.of.Lemma2.4.3}), with the conditions as follow:
    \begin{align*}
        \omega_3 := \frac{2}{s}\left(\frac{1}{2}-\frac{1}{r_1}\right) \in [0, 1],\\
        \omega_4 := \frac{2}{s}\left(\frac{1}{2}-\frac{1}{r_2}+\frac{\sigma(q)}{2}\right) \in \left[\frac{\sigma(q)}{s}, 1\right]
    \end{align*}
    and
    \begin{align*}
        \frac{1}{q} = \frac{p-1}{r_1} + \frac{1}{r_2}.
    \end{align*}
      For these reasons, we choose $r_2 = 2$ and $$r_1 =  \frac{p-1}{\frac{1}{q}-\frac{1}{2}} > 2 \qquad \text{ due to } \qquad pq > 2,\, q \in (1, 2). $$
In a similar manner, one obtains the estimate (\ref{Es.of.Lemma2.4.6}) with $s \in (1,2)$ as follows:
\begin{align*}
    \||u_t(\tau,\cdot)|^q\|_{\dot{H}^{s-1}} &\lesssim \|u_t(\tau,\cdot)\|_{L^{r_3}}^{q-1} \|u_t(\tau,\cdot)\|_{\dot{H}^{s-1}_{r_4}}\\
    &\lesssim \|u_t(\tau,\cdot)\|_{L^{\alpha}}^{(q-1)(1-\omega_5)+1-\omega_6} \|u_t(\tau,\cdot)\|_{\dot{H}^{s}}^{(q-1)\omega_5+\omega_6}\\
    &\lesssim (1+\tau)^{-2q+\frac{1}{2}-\frac{s-1}{2}+q[\gamma_2(p)]^+} \|u_t\|_{X_{21}(T)}^q,
\end{align*}
provided that there exist parameters $r_3, r_4$ satisfying
$$
     \omega_5 := \frac{\frac{1}{\alpha}-\frac{1}{r_3}}{\frac{1}{\alpha}-\frac{1}{2}+\frac{s}{2}} \in [0, 1], \quad
     \omega_6 := \frac{\frac{1}{\alpha}-\frac{1}{r_4}+\frac{s-1}{2}}{\frac{1}{\alpha}-\frac{1}{2}+\frac{s}{2}} \in \left[\frac{s-1}{s}, 1\right] \quad\text{and}\quad
    \frac{1}{2} =\frac{q-1}{r_3} +\frac{1}{r_4}.
$$
From these, we can choose
\begin{align*}
    \frac{1}{r_4} = \frac{1}{2}-\varepsilon_2 &\text{ and } \frac{1}{r_3} = \frac{\varepsilon_2}{q-1}.
\end{align*}
In summary, we
have completed the proof of Lemma \ref{lemma2.4}.
\end{proof}

\begin{proposition}\label{Pro2.3}
    Under the assumptions of Theorem \ref{Theorem3}, the following estimates hold for all $(u,v) \in Z_2(T) \times X_{22}(T)$:
    \begin{align}
        \|\partial_t u^{\rm non}\|_{X_{21}(T)} &\lesssim \|v\|_{X_{22}(T)}^p, \label{Esti.Pro2.3.1}\\
        \|v^{\rm non}\|_{X_{22}(T)} &\lesssim \|u_t\|_{X_{21}(T)}^q. \label{Esti.Pro2.3.2}
    \end{align}
\end{proposition}

\begin{proof}
\textbf{First, we prove the estimate (\ref{Esti.Pro2.3.1})}. We fix $\mu \in \{2, \alpha\}$. Using Lemma \ref{lemma2.4} and Lemma \ref{LinearEstimates_2D} with $\rho_2 = 1,\, \rho_1 = \eta, (d_1, d_2) =(0,0)$ for $\tau \in [0, t/2]$ and $\rho_1 = \rho_2=\eta, (d_1, d_2) = (0,0)$,  we obtain the following estimates:
        \begin{align}
        \|\partial_t u^{\rm non}(\tau,\cdot)\|_{L^{\eta}} &\lesssim \int_0^{t/2} (1+t-\tau)^{-2+\frac{1}{\eta}} \||v(\tau,\cdot)|^p\|_{L^1 \cap L^{\eta} \cap \dot{H}^{\sigma(\eta) }_{\eta}} d\tau \notag\\
        &\hspace{2cm}+ \int_{t/2}^t (1+t-\tau)^{-1} \||v(\tau,\cdot)|^p\|_{L^{\eta} \cap \dot{H}^{\sigma(\eta)}_{\eta}} d\tau\notag\\
        &\lesssim \bigg((1+t)^{-2+\frac{1}{\eta}}\int_0^{t/2}(1+\tau)^{-p+1} d\tau + (1+t)^{-p+\frac{1}{\eta}} \int_{t/2}^t (1+t-\tau)^{-1} d\tau \bigg) \|v\|_{X_{22}(T)}^p\notag\\
        &\lesssim (1+t)^{-2+\frac{1}{\eta}+[\gamma_2(p)]^+} \|v\|_{X_{22}(T)}^p, \label{Main.Esti.1}
    \end{align}
    where 
    \begin{align*}
        \sigma(\eta) = 
        \begin{cases}
            \frac{1}{q}-\frac{1}{2} &\text{ if } \eta = \alpha = q \in (1,2),\\
            0 &\text{ if } q \geq 2, \text{ that is, } \eta = \alpha =2.       \end{cases}
    \end{align*}
    Moreover, thanks to again Lemma \ref{lemma2.4} and Lemma \ref{LinearEstimates_2D} with $\rho_2 = 1, \rho_1 = 2, (d_1, d_2) = (s, 0)$ for $\tau\in [0, t/2)$, and $\rho_2 = \rho_1 = 2, (d_1, d_2)= (s, s)$ for $\tau \in [t/2, t]$ we also estimate
    \begin{align}
        \|\partial_t u^{\rm non}(t,\cdot)\|_{\dot{H}^{s}} &\lesssim \int_0^{t/2} (1+t-\tau)^{-\frac{3}{2}-\frac{s}{2}} \||v(\tau,\cdot)|^p\|_{L^1 \cap \dot{H}^{s}} d\tau \notag\\
        &\hspace{2cm}+ \int_{t/2}^t (1+t-\tau)^{-1} \||v(\tau,\cdot)|^p\|_{ \dot{H}^{s}} d\tau \notag\\
        &\lesssim \|v\|_{X_{22}(T)}^p \bigg((1+t)^{-\frac{3}{2}-\frac{s}{2}} \int_0^{t/2} (1+\tau)^{-p+1} d\tau\notag\\
        &\hspace{4cm} + (1+t)^{-p+\frac{1}{2} -\frac{s}{2}+\frac{\varepsilon_2}{2}} \int_{t/2}^t (1+t-\tau)^{-1} d\tau\bigg) \notag\\
        &\lesssim (1+t)^{-\frac{3}{2}-\frac{s}{2}+[\gamma_2(p)]^+} \|v\|_{X_{22}(T)}^p. \label{Main.Esti.2}
    \end{align}
From this, we can easily see that the estimates (\ref{Main.Esti.1}) and (\ref{Main.Esti.2}) lead to (\ref{Esti.Pro2.3.1}).

\textbf{Next, we prove the estimate (\ref{Esti.Pro2.3.2})}. Applying again Lemma \ref{lemma2.4} and Lemma \ref{LinearEstimates_2D} with $\rho_2 = 1,\, \rho_1 = \alpha_2$ and $(d_1, d_2) = (0,0)$ for all $\tau \in [0, t]$, we can also estimate the norms of 
$v^{\rm non}$ as follows, with $\alpha_2 \in \{2, p\}$ and $\sigma(\alpha_2) = \left|\frac{1}{\alpha_2}-\frac{1}{2}\right| < 1$:
\begin{align*}
    \|v^{\rm non}(t,\cdot)\|_{L^{\alpha_2}} &\lesssim \int_0^t (1+t-\tau)^{-1+\frac{1}{\alpha_2}} \||u_t(\tau,\cdot)|^q\|_{L^1 \cap L^{\alpha_2}} d\tau\\
    &\lesssim \|u_t\|_{X_{21}(T)}^q \int_0^t (1+t-\tau)^{-1+\frac{1}{\alpha_2}}(1+\tau)^{-2q+1+q[\gamma_2(p)]^+} d\tau\\
    &\lesssim (1+t)^{-1+\frac{1}{\alpha_2}} \|u_t\|_{X_{21}(T)}^q.
    \end{align*}
   Moreover, in the final estimate, it is easy to see that the condition  (\ref{condition3.1})  implies that
    \begin{align*}
        -2q+1+q[\gamma_2(p)]^+ < -1.
    \end{align*}
    Finally, thanks to again Lemma \ref{LinearEstimates_2D} with $\rho_2 = 1, \rho_1 = 2, (d_1, d_2) = (s, 0)$ for $\tau \in [0, t/2)$, and $\rho_2 = \rho_1 = 2, (d_1, d_2) = (s, 0)$ for $\tau \in [t/2, t]$, one has
    \begin{align*}
    \|v^{\rm non}(t,\cdot)\|_{\dot{H}^{s}} &\lesssim \int_0^{t/2} (1+t-\tau)^{-\frac{1}{2}-\frac{s}{2}} \||u_t(\tau,\cdot)|^q\|_{L^1 \cap \dot{H}^{s-1}} d\tau \\
    &\hspace{3cm} +\int_{t/2}^t (1+t-\tau)^{-\frac{s}{2}} \||u_t(\tau,\cdot)|^q\|_{L^2 \cap \dot{H}^{s-1}} d\tau\\
    &\lesssim \|u_t\|_{X_{21}(T)}^p \bigg((1+t)^{-\frac{1}{2}-\frac{s}{2}} \int_0^{t/2} (1+\tau)^{-2q+1+q[\gamma_2(p)]^+} d\tau\\
    &\hspace{4cm}+ (1+t)^{-2q+\frac{1}{2}+q[\gamma_2(p)]^+ }\int_{t/2}^t (1+t-\tau)^{-\frac{s}{2}}  d\tau\bigg)\\
    &\lesssim (1+t)^{-\frac{1}{2}-\frac{s}{2}} \|u_t\|_{X_{21}(T)}^q.
\end{align*}
From the estimates above, we have completed the proof of Proposition \ref{Pro2.3}.
    
\end{proof}

\begin{proposition}\label{Pro2.7}
    Under the assumptions of Theorem \ref{Theorem3}, the following estimates hold for all $(u,v)$ and $(\bar{u}, \bar{v}) \in Z_2(T) \times X_{22}(T)$:
     \begin{align}
     \|\mathcal{N}_{22}[v]-\mathcal{N}_{22}[\bar{v}]\|_{Y_{22}(T)} &\lesssim \|u_t-\bar{u}_t\|_{Y_{21}(T)}\left(\|u_t\|_{X_{21}(T)}^{q-1} + \|\bar{u}_t\|_{X_{21}(T)}^{q-1}\right).\label{Esti.Pro2.4.1}\\
\|\partial_t \mathcal{N}_{21}[u] - \partial_t \mathcal{N}_{21}[\bar{u}]\|_{Y_{21}(T)} &\lesssim \|v-\bar{v}\|_{Y_{22}(T)} \left(\|v\|_{X_{22}(T)}^{p-1} +\|\bar{v}\|_{X_{22}(T)}^{p-1} \right), \label{Esti.Pro2.4.2}
    \end{align}
\end{proposition}
\begin{proof}
    To begin with, it is a fact that
    \begin{align*}
    (\mathcal{N}_{22}[v]-\mathcal{N}_{22}[\bar{v}])(t,x) &= (v^{\rm non}- \bar{v}^{\rm non})(t,x) \\
       &= \int_0^t \mathcal{K}(t-\tau,x) \ast_x \left(|u_t(\tau,x)|^q-|\bar{u}_t(\tau,x)|^q\right) d\tau,\\
       (\partial_t \mathcal{N}_{21}[u] - \partial_t \mathcal{N}_{21}[\bar{u}])(t,x) &= (\partial_t u^{\rm non} -\partial_t \bar{u}^{\rm non})(t,x) \\
       &= \int_0^t \partial_t \mathcal{K}(t-\tau,x)\ast_x \left(|v(\tau,x)|^p - |\bar{v}(\tau,x)|^p\right) d\tau.
    \end{align*}
    \textbf{Now, we are going to prove the estimate (\ref{Esti.Pro2.4.1})}. Applying Lemma \ref{LinearEstimates_2D} with $ \rho_2 = m_3 \in (1, 2), \rho_1 = \beta = \max\{2,p\}, (d_1, d_2)= (0,0)$ for all $\tau \in [0,t]$, we derive the following estimates:
    \begin{align}
        \left\|(v^{\rm non}-\bar{v}^{\rm non})(t,\cdot)\right\|_{L^{\beta}} &\lesssim \int_0^t (1+t-\tau)^{-(\frac{1}{m_3}-\frac{1}{\beta})} \||u_t(\tau,\cdot)|^q -|\bar{u}_t(\tau,\cdot)|^q\|_{L^{m_3} \cap H^{-1}_\beta} d\tau \notag\\
        &\lesssim \int_0^t (1+t-\tau)^{-(\frac{1}{m_3}-\frac{1}{\beta})} \||u_t(\tau,\cdot)|^q -|\bar{u}_t(\tau,\cdot)|^q\|_{L^{m_3}} d\tau,\label{Ine2.5.1}
        \end{align}
        and with $\rho_2 = m_4 \in (1,2), \rho_1 = 2, (d_1, d_2) = (\theta(1-\kappa), 0)$ for all $\tau \in [0, t], \theta \in \{0,1\}$,  we also gain
        \begin{align}
        \left\|(v^{\rm non}-\bar{v}^{\rm non})(t,\cdot)\right\|_{\dot{H}^{\theta(1-\kappa)}} &\lesssim \int_0^t (1+t-\tau)^{-(\frac{1}{m_4}-\frac{1}{2})-\frac{\theta(1-\kappa)}{2}} \||u_t(\tau,\cdot)|^q -|\bar{u}_t(\tau,\cdot)|^q\|_{L^{m_4} \cap H^{-\kappa}} d\tau\notag\\
        &\lesssim \int_0^t (1+t-\tau)^{-(\frac{1}{m_4}-\frac{1}{2})-\frac{\theta(1-\kappa)}{2}} \||u_t(\tau,\cdot)|^q -|\bar{u}_t(\tau,\cdot)|^q\|_{L^{m_4}} d\tau. \label{Ine2.5.2}
    \end{align}
    Here, we used the embedding 
    \begin{align*}
        \|\psi\|_{L^\beta} \lesssim \|\psi\|_{H^1_{m_3}} &\,\,\text{ with }  \frac{1}{2} < \frac{1}{m_3} \leq \frac{1}{2} +\frac{1}{\beta},\\
        \|\psi\|_{L^2} \lesssim \|\psi\|_{H^{\kappa}_{m_4}} &\,\,\text{ with } \,\, \frac{1}{2} < \frac{1}{m_4} \leq \frac{1+\kappa}{2}  
    \end{align*}
    in the 2-dimentional space. Next, for $j = 3,4$, using H\"older's inequality with $1/m_j = 1/2 + 1/\mu_j$, we arrive at
    \begin{align*}
        \||u_t(\tau,\cdot)|^q -|\bar{u}_t(\tau,\cdot)|^q\|_{L^{m_j}} \leq \|u_t(\tau,\cdot)-\bar{u}_t(\tau,\cdot)\|_{L^2} \left(\|u_t(\tau,\cdot)\|_{L^{\mu_j(q-1)}}^{q-1} + \|\bar{u}_t(\tau,\cdot)\|_{L^{\mu_j(q-1)}}^{q-1} \right).
    \end{align*}
    The norm definition of $Y_{21}(T)$ leads to
    \begin{align*}
        \|u_t(\tau,\cdot)-\bar{u}_t(\tau,\cdot)\|_{L^2} \leq (1+\tau)^{-\frac{1}{2}-\frac{1}{\beta}+[\gamma_2(p)]^+} \|u_t-\bar{u}_t\|_{Y_{21}(T)}.
    \end{align*}
    Thanks to Proposition \ref{fractionalGagliardoNirenberg} to have
    \begin{align*}
        \|u_t(\tau,\cdot)\|_{L^{\mu_j(q-1)}}^{q-1}  &\lesssim \|u_t(\tau,\cdot)\|_{L^2}^{(q-1)(1-\omega_7)} \|u_t(\tau,\cdot)\|_{\dot{H}^{s}}^{(q-1)\omega_7}\\
        &\lesssim (1+\tau)^{-2(q-1)+\frac{1}{\mu_j}+(q-1)[\gamma_2(p)]^+} \|u_t\|_{X_{21}(T)}^{q-1},
    \end{align*}
    where
    \begin{align*}
        \omega_7 := \frac{2}{s}\left(\frac{1}{2}-\frac{1}{\mu_j(q-1)}\right) \in [0, 1],
    \end{align*}
    that is 
$$
            \displaystyle\frac{1}{\mu_j} \leq \displaystyle\frac{q-1}{2}\quad \text{ and }\quad
            \displaystyle\frac{1}{2} -\displaystyle\frac{1}{\mu_j(q-1)} \leq \displaystyle\frac{s}{2}.
$$
    We summarize the following conditions of $m_3, m_4$ with $\beta =\max\{2,p\}$:
  \begin{align*}
     \frac{1}{2} <  \frac{1}{m_3} \leq \min \left\{\frac{1}{2} + \frac{1}{\beta}, \frac{q}{2} \right\} \quad\text{ and }\quad \frac{1}{2} < \frac{1}{m_4} \leq \frac{1+\kappa}{2}.
  \end{align*}
  We can easily see that the parameters $m_3$ and $m_4$
  exist. In particular, we get
  \begin{align*}
      \frac{1}{m_3} = \frac{1}{m_4} = \frac{1+\kappa}{2}.
  \end{align*}
  For this reason, we have the following estimates:
  \begin{align*}
      &\||u_t(\tau,\cdot)|^q -|\bar{u}_t(\tau,\cdot)|^q\|_{L^{m_j}} \\
      &\hspace{1cm}\lesssim (1+\tau)^{-2(q-1) -1 -\frac{1}{\beta}+\frac{1}{m_j} + q[\gamma_2(p)]^+ } \|u_t-\bar{u}_t\|_{Y_{21}(T)} \left(\|u_t\|_{X_{21}(T)}^{q-1} + \|\bar{u}_t\|_{X_{21}(T)}^{q-1}\right).
  \end{align*}
  As a result from (\ref{Ine2.5.1})-(\ref{Ine2.5.2}), one finds
  \begin{align*}
      \left\|(v^{\rm non}-\bar{v}^{\rm non})(t,\cdot)\right\|_{L^{\beta}} &\lesssim P_1(t) \|u_t-\bar{u}_t\|_{Y_{21}(T)} \left(\|u_t\|_{X_{21}(T)}^{q-1} + \|\bar{u}_t\|_{X_{21}(T)}^{q-1}\right),\\
      \left\|(v^{\rm non}-\bar{v}^{\rm non})(t,\cdot)\right\|_{\dot{H}^{\theta(1-\kappa)}} &\lesssim P_2(t) \|u_t-\bar{u}_t\|_{Y_{21}(T)} \left(\|u_t\|_{X_{21}(T)}^{q-1} + \|\bar{u}_t\|_{X_{21}(T)}^{q-1}\right),
  \end{align*}
   where
  \begin{align*}
      P_1(t) &:= \int_0^t (1+t-\tau)^{-(\frac{1}{m_3}-\frac{1}{\beta})}  (1+\tau)^{-2(q-1) -1 -\frac{1}{\beta}+\frac{1}{m_3} + q[\gamma_2(p)]^+ } d\tau\\
      &\lesssim (1+t)^{-(\frac{1}{m_3}-\frac{1}{\beta})} \int_0^{t/2} (1+\tau)^{-2(q-1) -1 -\frac{1}{\beta}+\frac{1}{m_3} + q[\gamma_2(p)]^+} d\tau\\
      &\qquad+ (1+t)^{-2(q-1) -1 -\frac{1}{\beta}+\frac{1}{m_3} + q[\gamma_2(p)]^+} \int_{t/2}^t (1+t-\tau)^{-(\frac{1}{m_3}-\frac{1}{\beta})} d\tau\\
      &=: P_{11}(t)+P_{12}(t),\\
      P_2(t)&:= \int_0^t (1+t-\tau)^{-(\frac{1}{m_4}-\frac{1}{2})-\frac{\theta(1-\kappa)}{2}} (1+\tau)^{-2(q-1) -1 -\frac{1}{\beta}+\frac{1}{m_4} + q[\gamma_2(p)]^+} d\tau\\
      &\lesssim (1+t)^{-(\frac{1}{m_4}-\frac{1}{2})-\frac{\theta(1-\kappa)}{2}} \int_0^{t/2} (1+\tau)^{-2(q-1) -1 -\frac{1}{\beta}+\frac{1}{m_4} + q[\gamma_2(p)]^+} d\tau\\
      &\qquad + (1+t)^{-2(q-1) -1 -\frac{1}{\beta}+\frac{1}{m_4} + q[\gamma_2(p)]^+} \int_{t/2}^t (1+t-\tau)^{-(\frac{1}{m_4}-\frac{1}{2})-\frac{\theta(1-\kappa)}{2}} d\tau\\
      &=: P_{21}(t) + P_{22}(t).
  \end{align*}
To estimate $P_{11}(t)$, we consider the following cases.
  \begin{itemize}
  [leftmargin=*]
      \item \textbf{Case 1:} If $$-2(q-1) -1 -\frac{1}{\beta}+\frac{1}{m_3} + q[\gamma_2(p)]^+ \leq -1,$$ we immediately conclude
      $$P_{11}(t) \lesssim (1+t)^{-(\frac{1}{m_3}-\frac{1}{\beta})} \log(e+t) \lesssim 1.$$
      \item \textbf{Case 2:} If $$-2(q-1) -1 -\frac{1}{\beta}+\frac{1}{m_3} + q[\gamma_2(p)]^+ > -1,$$ we see that
      $$P_{11}(t) \lesssim (1+t)^{-2(q-1)+q[\gamma_2(p)]^+} \lesssim 1,$$ due to $pq > 2$, that is
      $$-2(q-1)+q[\gamma_2(p)]^+ < 0.$$
  \end{itemize}
  From this, we have $P_{11}(t) \lesssim 1$ for all $t > 0$. By proceeding similarly, we also obtain $P_{12}(t) \lesssim 1$ for all $t > 0$ due to
  $$-\frac{1}{m_3}+\frac{1}{\beta} > -1.$$ These lead to $P_1(t) \lesssim 1$ for all $t > 0$.\\
  For the term $P_{21}(t)$, we also divide two cases as follows:
  \begin{itemize}
  [leftmargin=*]
      \item \textbf{Case 1: } If $$-2(q-1) -1 -\frac{1}{\beta}+\frac{1}{m_4} + q[\gamma_2(p)]^+ \leq -1,$$
      we immediately obtain
      \begin{align*}
          P_{21}(t) &\lesssim (1+t)^{-(\frac{1}{m_4}-\frac{1}{2})-\frac{\theta(1-\kappa)}{2}} \log(e+t) = (1+t)^{-\frac{\kappa}{2}-\frac{\theta(1-\kappa)}{2}} \log(e+t) \lesssim (1+t)^{\frac{1}{2}-\frac{1}{\beta} -\frac{\theta(1-\kappa)}{2}}. 
      \end{align*}

  \item \textbf{Case 2: } If $$-2(q-1) -1 -\frac{1}{\beta}+\frac{1}{m_4} + q[\gamma_2(p)]^+ > -1,$$
  we arrive at
  \begin{align*}
      P_{21}(t) \lesssim (1+t)^{-2(q-1)+q[\gamma_2(p)]^+ - (\frac{1}{\beta}-\frac{1}{2})-\frac{\theta(1-\kappa)}{2}} \lesssim (1+t)^{\frac{1}{2}-\frac{1}{\beta} -\frac{\theta(1-\kappa)}{2}}.
  \end{align*}
  From these, we can conclude $P_{21}(t) \lesssim (1+t)^{\frac{1}{2}-\frac{1}{\beta} -\frac{\theta(1-\kappa)}{2}}$ for all $t >0$ and $\theta\in \{0,1\}.$ 
  \end{itemize}
  Finally, we estimate the quantity $P_{22}(t)$. It is a fact that
  \begin{align*}
      -\left(\frac{1}{m_4}-\frac{1}{2}\right) - \frac{\theta(1-\kappa)}{2} > -1.  
      \end{align*}
   Therefore, we see that
      \begin{align*}
          P_{22}(t) \lesssim (1+t)^{-2(q-1)+q[\gamma_2(p)]^+ - (\frac{1}{\beta}-\frac{1}{2})-\frac{\theta(1-\kappa)}{2}} \lesssim (1+t)^{\frac{1}{2}-\frac{1}{\beta} -\frac{\theta(1-\kappa)}{2}}.
      \end{align*}

  In summary, we can conclude that $P_2(t) \lesssim (1+t)^{\frac{1}{2}-\frac{1}{\beta} -\frac{\theta(1-\kappa)}{2}}$. The above calculations lead to the consequence that
  \begin{align}\label{Main.Es.Pro2.7.1}
      \left\|(v^{\rm non}-\bar{v}^{\rm non})(t,\cdot)\right\|_{L^{\beta}} &\lesssim  \|u_t-\bar{u}_t \|_{Y_{21}(T)} \left(\|u_t\|_{X_{21}(T)}^{q-1} + \|\bar{u}_t\|_{X_{21}(T)}^{q-1}\right)
  \end{align}
  and
  \begin{align}\label{Main.Es.Pro2.7.2}
      &\left\|(v^{\rm non}-\bar{v}^{\rm non})(t,\cdot)\right\|_{\dot{H}^{\theta(1-\kappa)}} \lesssim (1+t)^{\frac{1}{2}-\frac{1}{\beta} -\frac{\theta(1-\kappa)}{2}} \|u_t-\bar{u}_t\|_{Y_{21}(T)} \left(\|u_t\|_{X_{21}(T)}^{q-1} + \|\bar{u}_t\|_{X_{21}(T)}^{q-1}\right),
  \end{align}
  for all $t > 0$ and $\theta \in \{0,1\}$. The estimates (\ref{Main.Es.Pro2.7.1}) and (\ref{Main.Es.Pro2.7.2}) imply (\ref{Esti.Pro2.4.1}).

  \textbf{To complete the proof, we show the  proof of the estimate (\ref{Esti.Pro2.4.2}).} Thanks to again Lemma \ref{LinearEstimates_2D} for $1<m<2$, we have
  \begin{align*}
      \left\|\partial_t(u^{\rm non}-\bar{u}^{\rm non})(t,\cdot)\right\|_{L^2} &\lesssim \int_0^t (1+t-\tau)^{-(\frac{1}{m}-\frac{1}{2})-1} \||v(\tau, \cdot)|^p-|\bar{v}(\tau,\cdot)|^p\|_{L^m \cap L^2} d\tau.
  \end{align*}
  Thanks to again H\"older's inequality, we get 
  \begin{align*}
      \||v(\tau, \cdot)|^p-|\bar{v}(\tau,\cdot)|^p\|_{L^\theta} \leq \|v(\tau,\cdot)-\bar{v}(\tau,\cdot)\|_{L^{p\theta}} \left(\|v(\tau,\cdot)\|_{L^{p\theta}}^{p-1} + \|\bar{v}(\tau,\cdot)\|_{L^{p\theta}}^{p-1}\right),
  \end{align*}
  with $\theta \in \{m, 2
  \}$. Applying again Proposition \ref{fractionalGagliardoNirenberg}, we derive the following estimate:
  \begin{align*}
      \|v(\tau,\cdot)-\bar{v}(\tau,\cdot)\|_{L^{p\theta}}  &\lesssim \|v(\tau,\cdot)-\bar{v}(\tau,\cdot)\|_{L^\beta}^{1-\omega_8} \|v(\tau,\cdot)-\bar{v}(\tau,\cdot)\|_{\dot{H}^{1-\kappa}}^{\omega_8}\\
      &\lesssim (1+\tau)^{-\frac{1}{\beta}+\frac{1}{p\theta}}\|v- \bar{v}\|_{Y_{22}(T)},
  \end{align*}
  where $\theta \in \{m,2\}$ and
  \begin{align}\label{constrain2.7.3}
      \omega_8:= \frac{\frac{1}{\beta}-\frac{1}{p\theta }}{\frac{1}{\beta}-\frac{1}{2}+\frac{1-\kappa}{2}} \in [0,1], \text{ that is, }  \frac{\beta}{p} = \frac{\max\{2,p\}}{p} \leq m < 2.
  \end{align}
  On the other hand,
\begin{align*}
    \|v(\tau,\cdot)\|_{L^{p\theta}} &\lesssim \|v(\tau,\cdot)\|_{L^p}^{1-\omega_9} \|v(\tau,\cdot)\|_{\dot{H}^{s}}^{\omega_9}\lesssim (1+\tau)^{-1+\frac{1}{p\theta}} \|v\|_{X_{22}(T)}
\end{align*}
and
\begin{align*}
    \|\bar{v}(\tau,\cdot)\|_{L^{p\theta}} &\lesssim \|\bar{v}(\tau,\cdot)\|_{L^p}^{1-\omega_9} \|\bar{v}(\tau,\cdot)\|_{\dot{H}^{s}}^{\omega_9} \lesssim (1+\tau)^{-1+\frac{1}{p\theta}} \|\bar{v}\|_{X_{22}(T)},
\end{align*}
where
\begin{align*}
    \omega_9:= \frac{\frac{1}{p}-\frac{1}{p\theta}}{\frac{1}{p}-\frac{1}{2}+\frac{s}{2}} \in [0, 1].
\end{align*}
For these reasons, we arrive at
\begin{align*}
    \||v(\tau, \cdot)|^p-|\bar{v}(\tau,\cdot)|^p\|_{L^\theta} 
    &\lesssim (1+\tau)^{-p+1-\frac{1}{\beta}+\frac{1}{\theta}} \|v- \bar{v}\|_{Y_{22}(T)}\left(\|v\|_{X_{22}(T)}^{p-1}+\|\bar{v}\|_{X_{22}(T)}^{p-1}\right)\\
    &\lesssim (1+\tau)^{-p+1-\frac{1}{\beta}+\frac{1}{m}} \|v- \bar{v}\|_{Y_{22}(T)}\left(\|v\|_{X_{22}(T)}^{p-1}+\|\bar{v}\|_{X_{22}(T)}^{p-1}\right).
\end{align*}
  Therefore, we can conclude that
  \begin{align*}
      \left\|\partial_t(u^{\rm non}-\bar{u}^{\rm non})(t,\cdot)\right\|_{L^2} \lesssim P_3(t) \|v- \bar{v}\|_{Y_{22}(T)}\left(\|v\|_{X_{22}(T)}^{p-1}+\|\bar{v}\|_{X_{22}(T)}^{p-1}\right),
  \end{align*}
  where
  \begin{align*}
      P_3(t) &:= \int_0^t (1+t-\tau)^{-(\frac{1}{m}-\frac{1}{2})-1} (1+\tau)^{-p+1-\frac{1}{\beta}+\frac{1}{m}}  d\tau\\
      &\quad\lesssim (1+t)^{-(\frac{1}{m}-\frac{1}{2})-1} \int_0^{t/2} (1+\tau)^{-p+1-\frac{1}{\beta}+\frac{1}{m}}  d\tau \\
      &\qquad + (1+t)^{-p+1-\frac{1}{\beta}+\frac{1}{m}}\int_{t/2}^t (1+\tau)^{-(\frac{1}{m}-\frac{1}{2})-1} d\tau\\
      &=: P_{31}(t) + P_{32}(t). 
  \end{align*}
  We fix $1/m = 1/2 +\varepsilon_2$, that is, $m= 2/(1+2\varepsilon_2)$, it satisfies the constrain (\ref{constrain2.7.3}), and continue by dividing into the following two cases.
  \begin{itemize}
  [leftmargin=*]
      \item In the case $p \leq 2$, that is, $\beta = 2$,
  we can easily see that
  \begin{align*}
      P_{31}(t) &\lesssim
      \begin{cases}
      \vspace{0.3cm}
          (1+t)^{-1-\varepsilon_2} \log(e+t) &\text{ if } -p+\frac{1}{2}+\frac{1}{m} \leq -1,\\
          (1+t)^{1-p} &\text{ if } -p+\frac{1}{2}+\frac{1}{m} > -1
      \end{cases}\\
      &\lesssim (1+t)^{-1+(2-p+\varepsilon_2)} = (1+t)^{-\frac{1}{\beta}-\frac{1}{2}+[\gamma_2(p)]^+}
  \end{align*}
  and
  \begin{align*}
      P_{32}(t) \lesssim (1+t)^{1-p+\varepsilon_2} = (1+t)^{-\frac{1}{\beta}-\frac{1}{2}+[\gamma_2(p)]^+}. 
  \end{align*}
  \item In the case $p > 2$, that is, $\beta = p$, we arrive at 
  \begin{align*}
      P_{31}(t) &\lesssim
      \begin{cases}
      \vspace{0.3cm}
          (1+t)^{-1-\varepsilon_2} \log(e+t) &\text{ if } -p+1-\frac{1}{p}+\frac{1}{m} \leq -1,\\
          (1+t)^{\frac{3}{2}-p-\frac{1}{p}} &\text{ if } -p+1-\frac{1}{p}+\frac{1}{m} > -1
      \end{cases}\\
      &\lesssim (1+t)^{-\frac{1}{p}-\frac{1}{2}} = (1+t)^{-\frac{1}{\beta}-\frac{1}{2}+[\gamma_2(p)]^+}
  \end{align*}
  and 
  \begin{align*}
      P_{32}(t) \lesssim (1+t)^{\frac{3}{2}-p-\frac{1}{p}+\varepsilon_2} \lesssim (1+t)^{-\frac{1}{p}-\frac{1}{2}} = (1+t)^{-\frac{1}{\beta}-\frac{1}{2}+[\gamma_2(p)]^+}.
  \end{align*}
\end{itemize}
In summary, we obtain the following estimate:
\begin{align*}
    P_3(t) \lesssim (1+t)^{-\frac{1}{\beta}-\frac{1}{2}+[\gamma_2(p)]^+}.
\end{align*}
Hence, we can conclude that
\begin{align}\label{Main.Es.Pro2.7.3}
    &\left\|\partial_t(u^{\rm non}-\bar{u}^{\rm non})(t,\cdot)\right\|_{L^2} \lesssim (1+t)^{-\frac{1}{\beta}-\frac{1}{2}+[\gamma_2(p)]^+} \|v- \bar{v}\|_{Y_{22}(T)}\left(\|v\|_{X_{22}(T)}^{p-1}+\|\bar{v}\|_{X_{22}(T)}^{p-1}\right).
\end{align}
This immediately leads to the estimate (\ref{Esti.Pro2.4.2}).
Thus, we complete the proof of Proposition \ref{Pro2.7}.
\end{proof}

\textbf{Proof of Theorem \ref{Theorem3}.}
Employing again Lemma \ref{LinearEstimates_2D} and the definitions of $u^{\rm lin}, \, v^{\rm lin}$, we immediately obtain
\begin{align*}
    \|(\partial_t u^{\rm lin}, v^{\rm lin})\|_{X_2(T)} \lesssim \varepsilon\|(u_0, u_1, v_0, v_1)\|_{\mathcal{D}_{\alpha,p}^2(s)}.
\end{align*}
For this reason, combining with Propositions \ref{Pro2.3} and \ref{Pro2.7}, we obtain the following two important estimates for all $(u,v)$ and $(\bar{u}, \bar{v}) \in Z_2(T) \times X_{22}(T)$:
    \begin{align*}
        &\|(\partial_t \mathcal{N}_{21}[u],\, \mathcal{N}_{22}[v])\|_{X_2(T)} \leq C_1 \varepsilon \|(u_0, u_1)\|_{\mathcal{D}_{\alpha,p}^2(s)} + C_1\|(u_t, v)\|_{X_2(T)}^p + C_1\|(u_t, v)\|_{X_2(T)}^q, 
        \end{align*}
        and
        \begin{align*}
        &\|(\partial_t \mathcal{N}_{21}[u],\, \mathcal{N}_{22}[v])- (\partial_t \mathcal{N}_{21}[\bar{u}], \mathcal{N}_{22}[\bar{v}])\|_{Y_2(T)}\\ 
        &\quad\leq C_2 \|(u_t, v)-(\bar{u}_t, \bar{v})\|_{Y_2(T)}\\
        &\hspace{2cm} \times \big(\|(u_t,v)\|_{X_2(T)}^{p-1} + \|(u_t, v)\|_{X_2(T)}^{q-1}+ \|(\bar{u}_t,\bar{v})\|_{X_2(T)}^{p-1} + \|(\bar{u}_t, \bar{v})\|_{X_2(T)}^{q-1}\big). 
    \end{align*}
Similar to proof of Theorem \ref{Theorem1}, we can complete the proof of Theorem \ref{Theorem3}.
\section{Blow-up results}\label{Proof of blow-up results}
\subsection{Preliminaries}
In this section, our main aim is to prove Theorem \ref{Theorem2}. To prove our result, we recall the definition of weak solution to (\ref{Main.Eq.1}).
\begin{definition} \label{defweaksolution_1}
Let $p>1$,  $T>0$. We say that $(u,v) \in \mathcal{C}^1([0,T), L^2) \times \mathcal{C}([0, T), L^2) $ is a local weak solution to (\ref{Main.Eq.1}) if for any function $\phi_j(t,x)= \eta_j(t) \varphi(x)$, where $\eta_j \in \mathcal{C}_0^{\infty}([0, \infty))$ with $j = 1,2$ and $\varphi \in \mathcal{C}_0^{\infty}(\mathbb{R}^n)$, it holds
\begin{align*}
\int_0^T \int_{\R^n}|v(t,x)|^p \phi_1(t,x)dxdt &= \int_0^T \int_{\R^n} (\partial_t^2 - \Delta + \partial_t)u(t,x) \phi_1(t,x)dxdt
\end{align*}
and
\begin{align*}
    \int_0^T \int_{\R^n}|u_t(t,x)|^q \phi_2(t,x)dxdt &= \int_0^T \int_{\R^n} (\partial_t^2 - \Delta + \partial_t)v(t,x) \phi_2(t,x)dxdt
\end{align*}
If $T= \ity$, we say that $(u, v)$ is a global weak solution to (\ref{Main.Eq.1}) in $\mathcal{C}^1([0,\infty), L^2) \times \mathcal{C}([0, \infty), L^2)$. 
\end{definition}
Let us introduce the test functions $\eta= \eta(t)$ and $\varphi=\varphi(x)$ satisfying the following properties:
\begin{align*}
&1.\quad \eta \in \mathcal{C}_0^\ity([0,\ity)) \text{ and }
\eta(t)=\begin{cases}
1 \quad \text{ for }0 \le t \le 1/2, \\
0 \quad \text{ for }t \ge 1,
\end{cases} & \nonumber \\
&2.\quad \varphi \in \mathcal{C}_0^\ity(\mathbb{R}^n) \text{ and }
\varphi(x)=\begin{cases}
1 \quad \text{ for } |x| \le 1/2, \\
0 \quad \text{ for } |x| \ge 1,
\end{cases} & \nonumber \\
&3.\quad \eta^{-\frac{h'}{h}}\big(|\eta'|^{h'}+|\eta''|^{h'}\big) \text{ and } \varphi^{-\frac{h'}{h}} |\Delta \varphi|^{h'}\text{ is bounded, } & \\
\end{align*}
where $h \in \{p,q\}$ and $h'$ is the conjugate of $h$. In addition, we suppose that $\eta(t)$ is a decreasing
function and that $\varphi(x)$ is a radial function fulfilling $\varphi(x) \leq \varphi(y)$ for any $|x| \geq |y|$.
 Moreover, we define the following test function for $R \geq 1$:
$$ \Phi_{R}(t,x):= \eta_{R}(t) \varphi_R(x), $$
where $\eta_{R}(t):= \eta(R^{-2}t)$ and $\varphi_R(x):= \varphi(R^{-1}x)$. Moreover, we define the functionals
\begin{align*}
 &\mathcal{I}_{R}:= \int_0^{\ity}\int_{\R^n}|v(t,x)|^p \Phi_{R}(t,x) dxdt= \int_{0}^{R^2}\int_{|x| \leq R}|v(t,x)|^p \Phi_{R}(t,x) dxdt, \\
 \mathcal{I}_{1,R}&:= \int_{R^2/2}^{R^2}\int_{|x| \leq R}|v(t,x)|^p \Phi_{R}(t,x) dxdt, \qquad \mathcal{I}_{2,R}:= \int_{0}^{R^2}\int_{R/2 \leq |x| \leq R}|v(t,x)|^p \Phi_{R}(t,x) dxdt
\end{align*}
and
\begin{align*}
 &\mathcal{J}_{ R}:= \int_0^{\ity}\int_{\mathbb{R}^n}|u_t(t,x)|^q \Phi_{R}(t,x) dxdt= \int_{0}^{R^2}\int_{|x| \leq R}|u_t(t,x)|^q \Phi_{R}(t,x) dxdt, \\
  \mathcal{J}_{1,R}&:= \int_{R^2/2}^{R^2}\int_{|x| \leq R}|u_t(t,x)|^p \Phi_{R}(t,x) dxdt, \qquad \mathcal{J}_{2,R}:= \int_{0}^{R^2}\int_{R/2 \leq |x| \leq R}|u_t(t,x)|^p \Phi_{R}(t,x) dxdt
\end{align*}

\subsection{Proof of Theorem \ref{Theorem2}} At first, we assume that $(u, v)= (u(t,x), v(t,x))$ is a global weak solution from $\mathcal{C}^1([0, \infty), L^2) \times \mathcal{C}([0,\infty), L^2)$ to (\ref{Main.Eq.1}). From Definition \ref{defweaksolution_1}, we replace the functions $\phi_1(t,x)$  and $\phi_2(t,x)$ by $\Phi_{R}(t,x)$. Then, we perform integration by parts to derive
\begin{align}
    & \int_0^{\infty}\int_{\mathbb{R}^n} |v(t,x)|^p \Phi_{R}(t, x) dxdt  +\varepsilon  \int_0^{\infty} \eta_R(t)dt \int_{\mathbb{R}^n} u_0(x) (-\Delta \varphi)(R^{-1}x)dx+\varepsilon\int_{\mathbb{R}^n} u_1(x)\varphi_R(x)dx\notag \\
    &\hspace{1cm}= \int_0^{\infty}\int_{\mathbb{R}^n} u_t(t,x) \left(-\partial_t \eta_{R}(t) \varphi_R(x)  - \Psi_{R}(t) \Delta \varphi_R(x) + \eta_{R}(t) \varphi_R(x)\right) dxdt
    \label{equation1.3.1}
\end{align}
and 
\begin{align}
    &\int_0^{\infty}\int_{\mathbb{R}^n} |u_t(t,x)|^q \Phi_{R}(t, x) dxdt +\varepsilon\int_{\mathbb{R}^n} (v_0(x)+v_1(x))\varphi_R(x)dx\notag \notag\\
    &\hspace{1cm}= \int_0^{\infty} \int_{\mathbb{R}^n} v(t,x) (\partial_t^2 -\Delta - \partial_t) \Phi_{R}(t,x) dx dt,\label{equation1.3.2}
\end{align}
where $\Psi_{R}(t) = \displaystyle\int_t^{\infty} \eta_{R}(s) ds$. We fix the following quantities:
\begin{align*}
    \rho_{1, R} &:=  \int_0^{\infty} \eta_R(t)dt \int_{\mathbb{R}^n} u_0(x) (-\Delta \varphi)(R^{-1}x)dx + \int_{\mathbb{R}^n} u_1(x)\varphi_R(x)dx,\notag\\
    \rho_{2, R} &:= \int_{\mathbb{R}^n} (v_0(x)+v_1(x))\varphi_R(x)dx.
\end{align*}
Applying H\"older's inequality with $1/p + 1/p' = 1$  and using the change of variables $\Tilde{t} = R^{-2} t$, $\Tilde{x} = R^{-1} x$, we may estimates as follows:
\begin{align*}
    &\int_0^{\infty}\int_{\mathbb{R}^n} \left|v(t,x) \left(\partial_t^2  -\partial_t-\Delta\right)\Phi_{R}(t,x)\right| dxdt \lesssim \left(\mathcal{I}_{1,R}^\frac{1}{p} +\mathcal{I}_{2,R}^{\frac{1}{p}}\right) R^{-2 +\frac{n+2}{p'}}.
\end{align*}
From these, combining with (\ref{equation1.3.2}) leads to
\begin{align}
     \mathcal{J}_{R} + \varepsilon \rho_{2, R} \lesssim \left(\mathcal{I}_{1,R}^{\frac{1}{p}} + \mathcal{I}_{2,R}^{\frac{1}{p}}\right) R^{-2+\frac{n+2}{p'}}. \label{Main.Es2.2.1}
\end{align}
On the other hand, employing again H\"older's inequality and the change of variables $\Tilde{t} = R^{-2} t$, $\Tilde{x} = R^{-1} x$, we estimate the right-hand side of (\ref{equation1.3.1}) as follows:
\begin{align*}
    &\int_0^{\infty} \int_{\mathbb{R}^n} |u_t(t,x) \partial_t \Phi_{R}(t,x)| dxdt = \int_{R^2/2}^{R^2} \int_{|x| \leq R} |u_t(t,x) \partial_t \Phi_{R}(t,x)| dxdt \lesssim \mathcal{J}_{1,R}^{\frac{1}{q}} R^{-2+\frac{n+2}{q'}}
\end{align*}
and
\begin{align*}
    &\int_0^{\infty} \int_{\mathbb{R}^n}\left|u_t(t,x) \Psi_{R}(t) \Delta \varphi_R(x)\right| dxdt =\int_0^{R^2} \int_{R/2 \leq |x| \leq R}\left|u_t(t,x) \Psi_{R}(t) \Delta \varphi_R(x)\right| dxdt \lesssim \mathcal{J}_{2,R}^{\frac{1}{q}} R^{\frac{n+2}{q'}},
\end{align*}
where we note that
\begin{align*}
    \left|\int_t^{\infty} \eta(s) ds\right|^{q'} |\eta(t)|^{\frac{-q'}{q}}=  \left|\int_t^{1} \eta(s) ds\right|^{q'} |\eta(t)|^{\frac{-q'}{q}} \leq (1-t)^{q'} \eta(t) \leq 1 \text{ for all } t \in \left(\frac{1}{2}, 1\right).
\end{align*}
Moreover, one has
\begin{align*}
    \int_0^{\infty} \int_{\mathbb{R}^n} \left|u_t(t,x) \Phi_{R}(t,x)\right| dxdt \lesssim \mathcal{J}_{R}^{\frac{1}{q}} R^{\frac{n+2}{q'}}.
\end{align*}
 From this, we derive
\begin{align}
   & \mathcal{I}_{R} +\varepsilon \rho_{1, R} \lesssim \mathcal{J}_{R}^{\frac{1}{q}} R^{\frac{n+2}{q'}}. \label{Main.Es2.2.2}
\end{align}
Using the monotone convergence theorem, we derive
\begin{align*}
    \lim_{R \to \infty} \rho_{1, R} = \int_{\mathbb{R}^n} u_1(x) dx >0 \quad\text{ and } \quad\lim_{R \to \infty} \rho_{2, R} = \int_{\mathbb{R}^n} v_0(x) + v_1(x) dx > 0.
\end{align*}
Therefore, there exists a sufficiently large positive number $R_0$ such that
    \begin{align*}
        \rho_{1, R} > 0 \text{ and } \rho_{2, R} > 0,
    \end{align*}
    for all $R \in [R_0, \infty)$. Substituting the left side of
(\ref{Main.Es2.2.2}) into the right side of (\ref{Main.Es2.2.1}) we arrive at
\begin{align*}
    \mathcal{J}_{R} +  \varepsilon \rho_{2, R} \lesssim \mathcal{J}_{R}^{\frac{1}{pq}} R^{\frac{n+2}{pq'}-2+\frac{n+2}{p'}},
\end{align*}
that is
\begin{align}
    \varepsilon \rho_{2, R} \lesssim R^{\frac{pq}{pq-1}(\frac{n+2}{pq'}-2+\frac{n+2}{p'})}, \label{Main.Es.Theorem2}
\end{align}
for all $R \in [R_0, \infty)$.
Condition (\ref{condition2.2}) implies 
\begin{align*}
    \frac{n+2}{pq'} -2+\frac{n+2}{p'} < 0.
\end{align*}
 In the subcritical case, that is, $pq < 1 +2/n$,  taking $R \to \infty$, we arrive at
    \begin{align*}
        \int_{\mathbb{R}^n} \big(v_0(x) + v_1(x)\big) dx = 0.
    \end{align*}
    This is a contradiction to the assumption (\ref{condition2.1}). Thus, we have completed the proof of Theorem \ref{Theorem2}.

\section{Final Remarks}\label{Final Remarks}

\begin{remark}
\fontshape{n}
\selectfont
    In this paper, we have obtained the critical curve for problem (\ref{Main.Eq.1}) in one and two-dimensional spaces, namely $$pq =1 +\frac{2}{n}.$$ A natural question that arises is what the critical curve for problem (\ref{Main.Eq.1}) would be in higher dimensions 
$n \geq 3$. We believe that it remains the same curve $pq = 1+ 2/n$. However, due to certain limitations related to the high regularity of the initial data required in the estimates for solutions of the linear problem (\ref{Main.Eq.2}) ($\mu =0$), we are still facing challenges in proving this.
\end{remark}

\begin{remark}
\fontshape{n}
\selectfont
Based on the the approach used in this work, please fill a result for the following system:
\begin{equation*}
\begin{cases}
u_{tt} -\Delta u + u_t= |v_t|^p, &\quad x\in \R^n,\, t > 0, \\
v_{tt} -\Delta v +v_t = |u_t|^q, &\quad x \in \mathbb{R}^n, \,t > 0,\\
(u, u_t, v, v_t)(0,x)= \varepsilon (u_0, u_1, v_0, v_1)(x), &\quad x \in \mathbb{R}^n. \\
\end{cases}
\end{equation*}
Probably, one can expect that there always exists the global (in time) existence of small data solutions for any $p>1$ and $q>1$. This will be discussed in our forthcoming work.

\end{remark}

\appendix
\section{Some tools from Harmonic Analysis}\label{Sec.HarmonicAnalysis}
\begin{proposition}[Fractional Gagliardo-Nirenberg inequality, see Corollary 2.4 in \cite{Hajaiej2011}] \label{fractionalGagliardoNirenberg}
Let $1<p,\,p_0,\,p_1<\infty$, $a >0$ and $\theta\in [0,a)$. Then, it holds the following fractional Gagliardo-Nirenberg inequality:
$$ \|u\|_{\dot{H}^{\theta}_p} \lesssim \|u\|_{L^{p_0}}^{1-\omega(\theta,a)}\, \|u\|_{\dot{H}^{a}_{p_1}}^{\omega(\theta,a)}, $$
where $\omega(\theta,a) =\displaystyle\frac{\frac{1}{p_0}-\frac{1}{p}+\frac{\theta}{n}}{\frac{1}{p_0}-\frac{1}{p_1}+\frac{a}{n}}$ and $\displaystyle\frac{\theta}{a}\leq \omega(\theta,a) \leq 1$.
\end{proposition}

\medskip

\begin{proposition}[Fractional powers, see Proposition 42 and Corollary 43 in \cite{Duong2015}] \label{FractionalPowers}
Let $p>1$, $1< r <\infty$, where $s \in \big(n/r,p\big)$. Let us denote by $F(u)$ one of the functions $|u|^p,\, \pm |u|^{p-1}u$. Then, the following estimates hold:
$$\|F(u)\|_{H^{s}_r}\lesssim \|u\|_{H^{s}_r}\,\, \|u\|_{L^\infty}^{p-1} \quad \text{ and }\quad \| F(u)\|_{\dot{H}^{s}_r}\lesssim \|u\|_{\dot{H}^{s}_r}\,\, \|u\|_{L^\infty}^{p-1}. $$
\end{proposition}

\begin{proposition}[A fractional Sobolev embedding, see Corollary A.2 in \cite{Dao2019}] \label{Embedding}
Let  $1 < q < \infty$ and $0< s_1< n/q < s_2$. Then, it holds
$$ \|u\|_{L^\ity} \lesssim \|u\|_{\dot{H}^{s_1}_q}+ \|u\|_{\dot{H}^{s_2}_q}. $$
\end{proposition}

\begin{proposition}[Fractional chain rule, see Theorem 1.5 in \cite{Palmieri2018}]\label{chainrule}
    Let $s >0,\, p > \lceil s \rceil$ and $1 < r, r_1, r_2 < \infty$ satisfying the relation $$\frac{1}{r} = \frac{p-1}{r_1} + \frac{1}{r_2}.$$ Let us denote by $F(u)$ one of the functions $|u|^p, \pm |u|^{p-1}u$. Then, it holds
    \begin{align*}
        \| F(u)\|_{\dot{H}_r^s} \lesssim \|u\|_{L^{r_1}}^{p-1} \| u\|_{\dot{H}^s_{r_2}}.
    \end{align*}
\end{proposition}


\section*{Acknowledgments}
This research of Dinh Van Duong and Tuan Anh Dao is done during his period at Vietnam Institute for Advanced Study in Mathematics (VIASM) from June 2025 to August 2025. He would like to express a sincere thankfulness to VIASM for their hospitality and very kind support as well.



\begin{thebibliography}{00}
 \bibliographystyle{plain}
 \bibitem{Brezis2011} H. Brezis, \textit{Functional Analysis, Sobolev Spaces and Partial Differential Equations}, Universitext. Springer, New York, 2011.
 
 \bibitem{ChenDao2023} W. Chen, T.A. Dao, Sharp lifespan estimates for the weakly coupled system of
semilinear damped wave equations in the critical case, \textit{Math. Ann. } \textbf{385}(2023) 101-130.

\bibitem{ChenReissig2023} W. Chen, M. Reissig, On the critical exponent and sharp lifespan estimates for semilinear damped wave equations with data from Sobolev spaces of negative order , \textit{J. Evol. Equa. } (2023) 23:13.

\bibitem{DabbiccoEbert2017} M. D'Abbicco, M.R. Ebert, A new phenomenon in the critical exponent for structurally damped semi-linear evolution equations, \textit{Nonlinear Anal.}, \textbf{149} (2017), 1-40.

\bibitem{Dao2019} T.A.  Dao, M. Reissig, An application of $L^1$ estimates for oscillating integrals to parabolic like semi-linear structurally damped $\sigma$-evolution models, \textit{J. Math. Anal. Appl. } \textbf{15}(2019), 426-463.

\bibitem{Duong2015} P.T. Duong, M.K. Mezadek, M. Reissig, Global existence for semi-linear structurally damped $\sigma$-evolution
models, \textit{J. Math. Anal. Appl. } \textbf{431} (2015) 569–596.

\bibitem{DuongDao2025} D.V. Duong, T.A. Dao, On the critical exponent for semi-linear damped wave equations with power nonlinearity and initial data belonging to Sobolev spaces of negative order, preprint.

 \bibitem{Ebert2020} M.R. Ebert, G. Girardi, M. Reissig, Critical regularity of nonlinearities in semilinear classical damped
wave equations, \textit{ Math. Ann.} \textbf{378} (2020) 1311–1326.

\bibitem{Hajaiej2011} H. Hajaiej, L. Molinet, T. Ozawa, B. Wang, Necessary and sufficient conditions for the fractional Gagliardo-Nirenberg
inequalities and applications to Navier-Stokes and generalized boson equations, in: Harmonic Analysis and Nonlinear
Partial Differential Equations, in: RIMS Kokyuroku Bessatsu, vol. B26, Res. Inst. Math. Sci. (RIMS), Kyoto, 2011,
pp. 159–175.

\bibitem{Ikeda2019} M. Ikeda, T. Inui, M. Okamoto, Y. Wakasugi, $L^p - L^q$ estimates for the damped wave equation and the critical exponent for the nonlinear problem with slowly decaying data, \textit{Commun. Pure Appl. Anal.} \textbf{18}(2019) 1967-2008.

 \bibitem{IkedaOgawa2016} M. Ikeda, T. Ogawa, Lifespan of solutions to the damped wave equation with a critical nonlinearity, \textit{J. Differential Equations} \textbf{261} (2016) 1880–1903.

\bibitem{IkehataMiyaokaNakatake2004} R. Ikehata, Y. Miyaoka, T. Nakatake, Decay estimates of solutions for dissipative wave equationsin $\mathbb{R}^n$ with lower power nonlinearities, \textit{J. Math. Soc. Japan} \textbf{56}(2004) 365–373. 


\bibitem{IkehataOhta2002} R. Ikehata, M. Ohta, Critical exponents for semilinear dissipative wave equations in $\mathbb{R}^N$, \textit{J. Math. Anal. Appl.} \textbf{269} (2002)  87–97.

\bibitem{IkehataTanizawa2005} R. Ikehata, K. Tanizawa, Global existence of solutions for semilinear damped
wave equations in $\mathbb{R}^n$ with noncompactly supported
initial data, \textit{Nonlinear Anal.} \textbf{61}(2005) 1189-1208.

\bibitem{LiZhou1995} T.T. Li, Y. Zhou,  Breakdown of solutions to $\Box u + u_t = |u|^{1+\alpha}$, \textit{Discrete
Contin. Dyn. Syst. } \textbf{1}(1995) 503-520. 

\bibitem{LaiZhou2019} N.A. Lai, Y. Zhou, The sharp lifespan estimate for semilinear damped wave equation with Fujita critical power in higher dimensions, \textit{J. Math. Pures Appl.} \textbf{123}(2019) 229-243.


\bibitem{Matsumura1976} A. Matsumura, On the asymptotic behavior of solutions of semi-linear wave equations, \textit{Publ. Res. Inst. Math. Sci.} \textbf{12} (1976) 169–189. 

 \bibitem{MarcatiNishihara2003} P. Marcati, K. Nishihara, The $L^p-L^q$ estimates of solutions to one-dimensional damped wave equations and their application to the compressible flow
through porous media, \textit{J. Differential Equations} \textbf{191} (2003) 445–469.


\bibitem{Michihisa2021} H. Michihisa, $L^2$ asymptotic profiles of solutions to linear damped wave equations, \textit{J. Differential Equations}, \textbf{296} (2021) 573-592.

\bibitem{MitidieriPohozaev2001} E. Mitidieri, S.I. Pohozaev, A priori estimates and the absence of solutions of nonlinear partial differential equations and inequalities, \textit{Tr. Mat. Inst.
Steklova} \textbf{234} (2001) 1–384. 

\bibitem{Narazaki2004} T. Narazaki, $L^p-L^q$ estimates for damped wave equations and their applications to semilinear problem, \textit{J. Math. Soc. Japan. } \textbf{56}(2004),  586–626.

\bibitem{MitidieriPohozaev2009} E. Mitidieri, S.I. Pohozaev, Lifespan estimates for solutions of some evolution inequalities, \textit{Differ. Uravn.} \textbf{45} (2009) 1441–1451.


\bibitem{Nishihara2003} K. Nishihara, $L^p-L^q$ estimates of solutions to the damped wave
equation in $3$-dimensional space and their application, \textit{Math. Z.} \textbf{244} 
 (2003) 631-649.

\bibitem{NishiharaWakasugi2014} K. Nishihara, Y. Wakasugi, Critical exponent for the Cauchy problem to the weakly coupled damped wave system, \textit{Nonlinear Anal.} \textbf{108} (2014) 249-259.

 \bibitem{NishiharaWakasugi2015} K. Nishihara, Y. Wakasugi, Global existence of solutions for a weakly coupled system of semilinear damped wave equations, \textit{J. Differential Equations} \textbf{259}(2015) 4172-4201.

\bibitem{Palmieri2018} A. Palmieri, M. Reissig, Semi-linear wave models with power non-linearity and scale-invariant time-dependent mass and
dissipation, II, \textit{Math. Nachr. } \textbf{291} (2018) 1859–1892.

\bibitem{SunWang2007} F. Sun, M. Wang, Existence and nonexistence of global solutions for a nonlinear hyperbolic system
with damping, \textit{Nonlinear Anal. } \textbf{66} (2007)  2889–2910.

\bibitem{Takeda2009} H. Takeda, Global existence and nonexistence of solutions for a system of nonlinear
damped wave equations, \textit{J. Math. Anal. Appl.} \textbf{360} (2009) 631-650.


\bibitem{Takeda2015} H. Takeda, Higher-order expansion of solutions for a damped wave equation, \textit{Asymptotic Anal.}, \textbf{94} (2015) 1-31.

\bibitem{TodorovaYordanov2001} G. Todorova, B. Yordanov, Critical Exponent for a Nonlinear Wave Equation with Damping, \textit{J. Differential Equations} \textbf{174} (2001) 464-489.

\bibitem{Zhang2001} Q.S. Zhang, A blow-up result for a nonlinear wave equation with damping: the critical case, \textit{C. R.
Acad. Sci. Paris Sér. I Math. } \textbf{333}(2001),  109–114.
\end{thebibliography}
\end{document}